\documentclass{amsart}

\usepackage{amssymb,amsfonts,amsthm,amsmath,latexsym}
\usepackage[all]{xy}
\usepackage[T1]{fontenc}
\usepackage{mathrsfs}
\usepackage[utf8]{inputenc}
\usepackage[english]{babel}
\usepackage{amssymb}
\usepackage{amsthm}
\usepackage{graphics}
\usepackage{amsmath}
\usepackage{amstext}
\usepackage{multirow}
\usepackage{hyperref}
\usepackage{comment}
\usepackage{color}
\usepackage{tikz-cd}

\newtheorem{thm}{Theorem}[section]
\newtheorem{dfn}[thm]{Definition}
\newtheorem{lemma}[thm]{Lemma}
\newtheorem{prop}[thm]{Proposition}
\newtheorem{cor}[thm]{Corollary}

\newtheorem{THM}{Theorem}

\newtheorem{PROP}[THM]{Proposition}

\newtheorem{CONJ}[THM]{Conjecture}

\theoremstyle{remark}

\newtheorem{rem}[thm]{Remark}

\newcommand{\mb}{\mathbb}
\newcommand{\mc}{\mathcal}

\newcommand{\R}{\mb R}
\newcommand{\C}{\mb C}
\newcommand{\Hj}{\mb H}
\newcommand{\Pj}{\mb P}
\newcommand{\Z}{\mb Z}

\newcommand{\D}{\mb D}
\newcommand{\K}{\mathcal K}

\newcommand{\F}{\mc F}
\newcommand{\G}{\mc G}

\newcommand{\rato}{\dashrightarrow}

\DeclareMathOperator{\Aut}{Aut}
\DeclareMathOperator{\AutFix}{Aut^{\mathrm{fix}}}
\DeclareMathOperator{\Bir}{Bir}
\DeclareMathOperator{\BirFix}{Bir^{\mathrm{fix}}}
\DeclareMathOperator{\Bim}{Bim}
\DeclareMathOperator{\End}{End}
\DeclareMathOperator{\EndFix}{End^{\mathrm{fix}}}
\DeclareMathOperator{\alg}{alg}
\DeclareMathOperator{\trans}{tr}
\newcommand{\KF}{K_{\F}}
\newcommand{\KG}{K_{\G}}
\newcommand{\KX}{K_X}
\newcommand{\TF}{T_{\F}}
\newcommand{\NF}{N_{\F}}
\newcommand{\CNF}{N^*_{\F}}
\newcommand{\KNF}{K_{X/\F}}

\DeclareMathOperator{\WDiv}{WDiv}

\DeclareMathOperator{\divisor}{div}
\DeclareMathOperator{\Cl}{Cl}
\DeclareMathOperator{\discrep}{discrep}
\DeclareMathOperator{\Bl}{Bl}

\DeclareMathOperator{\PSL}{PSL}

\DeclareMathOperator{\im}{Im}
\DeclareMathOperator{\codim}{codim}

\DeclareMathOperator{\II}{II}

\DeclareMathOperator{\Pic}{Pic}

\DeclareMathOperator{\sing}{sing}
\DeclareMathOperator{\id}{id}

\DeclareMathOperator{\Aff}{Aff}

\DeclareMathOperator{\dev}{dev}

\DeclareMathOperator{\Id}{Id}
\DeclareMathOperator{\Hom}{Hom}

\DeclareMathOperator{\ord}{ord}
\DeclareMathOperator{\Alb}{Alb}
\DeclareMathOperator{\alb}{alb}

\numberwithin{equation}{section}

\begin{document}
\title{Rational endomorphisms of codimension one holomorphic foliations}
\author[Lo Bianco]{Federico Lo Bianco}
\address{Federico Lo Bianco \\Aix Marseille Univ\\
		CNRS, Centrale Marseille, I2M\\
		Marseille\\
		France}
\email{\href{mailto:federico.lo-bianco@univ-amu.fr}{federico.lo-bianco@univ-amu.fr}}

\author[Pereira]{Jorge Vitório Pereira}
\address{Jorge Vitório Pereira \\IMPA, Estrada Dona Castorina, 110, Horto\\ Rio de Janeiro,
Brasil}
\email{\href{mailto:jvp@impa.br}{jvp@impa.br}}

\author[Rousseau]{Erwan Rousseau}
\address{Erwan Rousseau \\ Institut Universitaire de France
	\& Aix Marseille Univ\\
		CNRS, Centrale Marseille, I2M\\
		Marseille\\
		France}
\email{\href{mailto:erwan.rousseau@univ-amu.fr}{erwan.rousseau@univ-amu.fr}}

\author[Touzet]{Fr\'ed\'eric Touzet}

\address{Fr\'ed\'eric Touzet \\ Univ Rennes, CNRS, IRMAR - UMR 6625, F-35000 Rennes, France.}
\email{\href{mailto:frederic.touzet@univ-rennes1.fr}{frederic.touzet@univ-rennes1.fr}}

\thanks{This work was supported by the ANR project \lq\lq FOLIAGE\rq\rq{}, ANR-16-CE40-0008  and CAPES-COFECUB Ma 932/19 project. The second author was supported by Cnpq and FAPERJ.}

\begin{abstract}
    In this work, we study dominant rational maps preserving singular holomorphic codimension one foliations on projective manifolds and that exhibit non-trivial transverse dynamics.
\end{abstract}

\maketitle
\sloppy
\setcounter{tocdepth}{1}
\tableofcontents

\section{Introduction}
In this article we investigate the rational symmetries of a singular holomorphic codimension one foliation.

\subsection{Earlier works}
To our knowledge, the first work dealing with  foliations invariant by rational maps is due to Brunella, who, answering a question by Cerveau, proved that no algebraic foliation on $\mathbb P^2$ is preserved by a non-elementary polynomial automorphism of $\C^2$. He obtained this result as a corollary of his study on minimal models for foliated surfaces, see \cite{MR1708643}.
Later, foliations on surfaces invariant by birational maps were studied in depth by Cantat and Favre in \cite{MR1998612}. They explored the classification of birational transformations of surfaces \cite{MR1867314} to give a precise description of the foliated surfaces with an infinite group of birational transformations.  Simultaneously,  it was proved in \cite{MR1957664} that foliations of maximal Kodaira dimension on projective surfaces have a finite group of birational transformations.

Foliations on projective surfaces invariant by non-invertible rational maps were classified in \cite{MR2818727} by exploring
the birational classification of foliations on projective surfaces according to their Kodaira dimension, by Brunella, McQuillan, and Mendes.

The analog problem for foliations on higher dimensional manifolds is much less studied. Some activity on the subject was spurred by a question posed  by Guedj in \cite[page 103]{MR2932434}: if $h$ is a bimeromorphic map of a compact Kähler manifold which is not cohomogically hyperbolic, is it true that $h$  preserves a foliation? A negative answer  was provided by Bedford, Cantat, and Kim in \cite{MR3277202}, by exhibiting families of pseudo-automorphisms of  rational $3$-folds which are not cohomologically hyperbolic and do not preserve any foliation. Their strategy explores the explicit form of the pseudo-automorphisms they construct, particularly the existence of certain invariant surfaces, and uses the available knowledge about foliations on surfaces invariant by rational maps.

\subsection{Transversely projective foliations}
In this work, we focus on the following question: given a projective manifold $X$ and a dominant rational map $f:X \dashrightarrow X$ preserving a codimension one foliation $\F$,  under which conditions some iterate of $f$ preserves each leaf of $\F$? When this happens, we say that the transverse action of $f$ is finite.  More generally, we ask under which conditions the semigroup of dominant rational maps of $X$ admits a
finite index subsemigroup that preserves each leaf of $\F$ or, more succinctly, under which conditions the transverse action of $\End(\F)$ is finite?

Our approach to this question was inspired by a conjecture of Cerveau and Lins Neto, which predicts that codimension one foliations on projective manifolds are either rational pull-backs of foliations on surfaces or admit a (singular) transversely projective structure in the sense of \S\ref{SS:transv proj}. A confirmation of this conjecture would reduce the study of rational maps with infinite transverse action on codimension one foliations to the transversely projective case.

\begin{THM}\label{THM:transv proj}
    Let $X$ be a projective manifold and let $\F$ be a transversely projective foliation of codimension one on $X$.
    If the transverse action of $\End(\F)$  is infinite then $\F$ is virtually transversely additive.
\end{THM}

As explained in \S\ref{SS:transv add}, a transversely additive foliation is a foliation defined by a closed $1$-form with coefficients in a finite algebraic extension of $\C(X)$.

\subsection{Transversely hyperbolic foliations}
The starting point of our proof of Theorem \ref{THM:transv proj} is the description of transversely projective foliations provided by \cite{MR3522824} which relies on Corlette-Simpson's classification of rank two representations of quasi-projective fundamental groups \cite{MR2457528}. When a transversely projective foliation is neither virtually transversely additive nor  a rational pull-back of a foliation on a surface, then the main result of \cite{MR3522824}
says that the monodromy representation of the transverse structure of $\F$ is Galois conjugate to the monodromy representation of a singular transversely hyperbolic foliation (see \S\ref{S:transv hyp} for a definition).

The technical core of our proof of Theorem \ref{THM:transv proj}, presented in Sections \ref{S:transv hyp} and \ref{S:THM A}, establishes several properties of transversely hyperbolic foliations on projective manifolds which might be of independent interest: transversely hyperbolic structures defined on the complement of a divisor  extend (perhaps singularly) through the divisor (Theorem \ref{TH:localextension}), their monodromy is Zariski dense (Proposition \ref{monodromy curve in Shimura}) and determine uniquely the foliation (Theorem \ref{T:Uniqueness}).

These properties are used to show that rational symmetries of transversely projective foliations induce symmetries of an associated transversely hyperbolic foliation. Furthermore, the monodromy of this transverse hyperbolic structure gives rise to a polarized variation of Hodge structures, according to Corlette-Simpson's results. We also use a recent result by Brunebarbe-Cadorel \cite{brunebarbe2017hyperbolicity} which implies that,
after factoring out the algebraic part of $\F$, the ambient manifold of the transversely hyperbolic foliation is of log-general type. These results can be put together to show that infinite transverse action implies the existence of an infinitesimal symmetry, what is impossible for transversely hyperbolic foliations.

\subsection{Zariski dense dynamics}
The statement of Theorem \ref{THM:transv proj} gives little information about the nature
of rational endomorphisms $f:X \dashrightarrow X$ preserving a codimension one foliation $\F$. In order to have a precise description of $f$, we restrict to the class of purely transcendental foliations (i.e. through a general point of a general leaf there is no positive dimensional algebraic subvariety tangent to it) and to rational endomorphisms admitting a Zariski dense orbit.

\begin{THM}\label{THM:core of structure}
    Let $X$ be a projective manifold and let $\F$ be a transversely projective and purely transcendental foliation of codimension one on $X$.
    If $f \in \End(X,\F)$ has a Zariski dense orbit then $f$  is conjugated to a Lattes-like map.
\end{THM}

A rational endomorphism $f: X \dashrightarrow X$ is a  Lattes-like map if there exists an abelian algebraic group $A$, a cyclic finite group $\Gamma$ acting on $A$, and a map $\varphi :A \to A$ which is the composition of a group endomorphisms with a translation and which factorizes as $\bar{\varphi}$, $\bar {\varphi}\circ p=p\circ\varphi$ through the projection $p: A \mapsto A/\Gamma$:
\begin{center}
	\begin{tikzcd}[column sep=small]
	A  \arrow[d,"p"] \arrow[rr, "\varphi" ] & & A \arrow[d,"p"]\\
	A/\Gamma \arrow[rr, "\bar{\varphi}"] & & A/\Gamma
	\end{tikzcd}
\end{center}
and such that $f$ is birationally conjugated to the
map $\bar{\varphi}$  induced by $\varphi$ on the quotient $A/\Gamma$.

Theorem \ref{THM:core of structure}, combined with  factorization results for foliations  and for rational endomorphisms (see Section \ref{S:factorization results}), provides a rough description of rational endomorphisms of transversely
projective foliations with infinite transverse action.  Its proof, which can be found in \S\ref{sec: Proof B}, relies on  reduction of singularities for foliations defined by closed rational $1$-forms (Proposition \ref{P:Seidenberg for closed forms}) and on the following result.

\begin{THM}\label{THM:qalbanese kod0}
    Let $f: X \dashrightarrow X$ be a rational map on a projective manifold $X$ with a Zariski dense orbit and let $D$ be a simple normal crossing
    divisor on $X$. If $f^*D$ has support contained in the support of $D$ then the quasi-Albanese morphism $\alb_{(X,D)} : X- D \to \Alb(X,D)$ is a  dominant  rational map with irreducible general fiber.
\end{THM}

Theorem \ref{THM:qalbanese kod0} is proven in \ref{SS:proofTHC}  and turns out to be a consequence of Kawamata's characterization of semi-abelian varieties \cite{MR622451}, and the linearity of the Zariski closure of orbits of endomorphisms of semi-abelian varieties \cite[Fact 2.9]{MR3937327}.

\subsection{A conjecture}
We do believe that the hypothesis on the transverse structure of $\F$ is not necessary, as predicted by Cerveau-Lins Neto  conjecture.

\begin{CONJ}\label{Conj:A}
    Let $X$ be a projective manifold and let $\F$ be a codimension one foliation on $X$.
    If the transverse action of $f \in \End(X,\F)$ is infinite then  $\F$ is virtually transversely additive.
\end{CONJ}

In order to verify Conjecture \ref{Conj:A}, it suffices to consider foliations invariant by rational endomorphisms with Zariski dense orbits, as shown in Corollary \ref{C:reduz}. In particular, Conjecture \ref{Conj:A} holds true when the Zariski closure of the general orbit of $f$
has dimension at most two, see Proposition \ref{P:orbit12}.

It is interesting to relate Conjecture \ref{Conj:A} to the problem of describing pairs of commuting rational endomorphisms. In the case of endomorphisms (no indeterminacies points) of $\mathbb P^n$, $n\ge 2$, the problem was studied by Dihn and Sibony in \cite{MR1931758}, under the extra assumptions that the degree of the iterates are all distinct, and without this assumption in the particular case of $\mathbb P^2$  by Kaufmann in \cite{MR3784253}.

Our next result confirms Conjecture \ref{Conj:A} for foliations on $\mathbb P^n$ invariant by endomorphisms of projective spaces
and can be seen as a (weak) analog of the results of Dihn-Sibony and Kaufmann.

\begin{PROP}\label{PROP:endo}
    If $\mathcal F$ is a codimension one foliation on a projective manifold $X$ with $\Pic(X) = \mathbb Z$ which is invariant by an endomorphism $f:X\to X$ of degree at least two then $\mathcal F$ is transversely additive.
\end{PROP}

It is conceivable that there exists an unified approach to attack both problems: foliations invariant by rational endomorphisms and commuting endomorphisms. Perhaps, one should look for a reformulation of both problems in terms of properties of Malgrange's groupoid of rational endomorphisms, see \cite{malgrange:hal-00469778}, \cite{MR2325017} and references therein.

\subsection{Foliations of (adjoint) general type}
Our last main result generalizes \cite[Theorem 1]{MR1957664} in two different directions: no restriction is imposed on the codimension of the foliation or on the dimension of the ambient variety, and the canonical divisor of the foliation can be replaced by any convex combination of it with the Weil divisor $\KNF$ associated to the determinant of the conormal sheaf of the foliation.

\begin{THM}\label{THM:adjoint}
    Let $\varepsilon \ge 0$ be a non-negative rational number, and let $\F$ be a foliation with
    $\varepsilon$-canonical singularities on a normal projective variety $X$.
    If $\KF + \varepsilon \KNF$ is big then the group of birational automorphisms of $\F$ is finite.
\end{THM}

In the terminology of \cite{MR3957403}, the group of birational transformations of a foliation of adjoint general type is finite.

One draw-back of the above result, is that it is unknown in dimensions strictly greater than $3$ if every foliation is birationally equivalent to a foliation with canonical singularities or $\varepsilon$-canonical singularities when $\varepsilon$ is small.  Anyway, in dimension three there are results by Cano \cite{MR2144971} (for codimension one foliations) and McQuillan-Panazzolo \cite{MR3128985} (for dimension one foliations) which guarantee the existence of a birational model with only canonical singularities.

Extrapolating the picture drawn by the classification of foliations on surfaces according to their adjoint dimension \cite{MR3957403}, it seems reasonable to expect  that purely transcendental codimension one foliations which are not of adjoint general type are transversely projective.
If this expectation is confirmed then our results would confirm Conjecture \ref{Conj:A}.

\subsection{Structure of the paper}
Section \ref{S:foliations} presents the basic definitions of the subject, recalls the results in dimension two concerning foliations with infinite semi-group of rational endomorphisms, and recalls the definition of algebraic and transcendental parts of a foliation.

Section \ref{S:Transversely} presents the definitions of (virtually) transversely additive, transversely affine, transversely projective, and transversely hyperbolic foliations as well as their basic properties.  Both  Sections \ref{S:foliations} and \ref{S:Transversely} contain no new results except for simple observations like Proposition \ref{P:reductiontotranscendental}.

Section \ref{S:orbits} recalls factorization results for semigroups of rational endomorphisms and shows how to reduce the study of infinite transverse actions of rational endomorphisms on foliations to the study of foliations invariant by rational maps with Zariski dense orbits, cf. Corollary \ref{C:reduz}.

Section  \ref{S:transv hyp} is devoted to the study of transversely hyperbolic foliations. It is perhaps the most technical part of the paper. It proves an extension result for transversely hyperbolic foliations (Theorem \ref{TH:localextension}), shows that transversely hyperbolic structures on quasi-projective manifolds always have Zariski dense monodromy (Proposition \ref{monodromy curve in Shimura}), and that this monodromy determines the foliation (\ref{TH:uniquemonodromy}).

The proof of Theorem \ref{THM:transv proj} is presented in Section \ref{S:THM A}. It builds on the description of the global structure of transversely projective foliations on projective manifolds proved in \cite{MR3522824} and recalled in Subsection \ref{SS:transv proj} and on our results on transversely hyperbolic foliations, more specifically Theorem \ref{TH:uniquemonodromy}.

Section \ref{S:qalbanese} studies the quasi-Albanese morphism of a projective manifold admitting a rational endomorphism with Zariski dense dynamics and contains the proof of Theorem \ref{THM:qalbanese kod0}. It can be read independently from the other sections and does not involve foliations. Theorem \ref{THM:core of structure} is proved in Section \ref{S:Zariski dense}.

Section \ref{S:endo} proves Proposition \ref{P:endogeral} providing evidence toward Conjecture \ref{Conj:A} and can also be read independently from the other sections. The same holds true for Section \ref{S:adjoint} which contains the proof of Theorem \ref{THM:adjoint}. We chose to include in this last Section a brief review of the basic notions/terminology used in the birational geometry  of foliations since in it, unlikely in remainder of the paper, we are lead to work with foliations on singular varieties.

\subsection{Acknowledgements} We thank Stefan Kebekus for providing a proof of Item (\ref{I:68}) of Lemma \ref{L:log log}.

\section{Singular holomorphic foliations}\label{S:foliations}

\subsection{Foliations}
A singular holomorphic foliation $\F$ on a compact complex  manifold $X$ is determined by a pair $T_{\F}\subset T_X$ and $N^*_{\F} \subset \Omega^1_X$ of coherent subsheaves such that
\begin{enumerate}
    \item\label{I:TF} $T_{\F}$ is the annihilator of  $N^*_{\F}$, meaning that $T_{\F}$ is the kernel of the morphism
    \begin{align*}
        T_X  & \longrightarrow \Hom(N^*_{\F} , \mathcal O_X ) \\
        v  & \longmapsto  ( \omega \mapsto \omega(v)) \, ; \text{ and }
    \end{align*}
    \item\label{I:CNF} $N^*_{\F}$ is the annihilator of $T_{\F}$,meaning that $N^*_{\F}$ is the kernel of the morphism
    \begin{align*}
        \Omega^1_X  & \longrightarrow \Hom(T_{\F} , \mathcal O_X ) \\
         \omega & \longmapsto  ( v \mapsto \omega(v)) \, ; \text{ and }
    \end{align*}
    \item\label{I:involutive} $T_{\F}$  is closed under Lie bracket, meaning that the O'Neil tensor
    \begin{align*}
        \bigwedge^2 T_{\F} & \longrightarrow \frac{T_X}{T_{\F}} \\
        v_1\wedge v_2 &\longmapsto [ v_1, v_2 ] \mod T_{\F}
    \end{align*}
    vanishes identically.
\end{enumerate}

It follows from Condition (\ref{I:TF}) that $T_{\F}$ is a reflexive subsheaf of $T_X$ and, moreover, that $T_{\F}$ is saturated in $T_X$, i.e. $T_X/T_{\F}$ is torsion-free. The sheaf $T_{\F}$ is called the tangent sheaf of $\F$. Similarly, Condition (\ref{I:CNF}) implies that $N^*_{\F}$ is a reflexive and saturated subsheaf of $\Omega^1_X$, called the conormal sheaf of $\F$. The dimension of $\F$ ($\dim(\F)$) is the rank of $T_{\F}$ and the codimension of $\F$ ($\codim(\F)$) is the rank of $N^*_{\F}$. The dual of $N^*_{\F}$ is the normal sheaf of $\F$ and we will denote it by $N_{\F}$. The dual of $T_{\F}$ is the cotangent sheaf of $\F$ and is denoted by $T^*_{\F}$ or $\Omega^1_{\F}$.

The smooth locus of $\F$ is, by definition, the locus where both $T_{\F}$ and the quotient sheaf $T_X/T_{\F}$ are locally free.
The complement of the smooth locus of $\F$ is called the singular locus of $\F$ and is denoted by $\sing(\F)$.

\subsection{Leaves and tangent subvarieties}
If $\F$ is a foliation of codimension $q$ then Condition (\ref{I:involutive}) allows us to apply Frobenius's integrability theorem to guarantee the existence of an open covering $\mathcal U$ of $X- \sing(\F)$ and of holomorphic maps of maximal rank and connected fibers $f_{U} : U \to \mathbb C^q$ such that ${T_{\F}}_{|U} = \ker df$ for any $U \in \mathcal U$. If one considers the smallest equivalence relation on $X- \sing(\F)$ which identifies points on an open subset $U \in \mathcal U$ with the same images under $f_U$ then the equivalence class of a point $x \in X- \sing(\F)$  is, by definition, the leaf of
$\F$ through $x$.

If $Y \subset X$ is an irreducible subvariety then we will say that $Y$ is tangent to a foliation $\F$ if there exists a dense open subset $U$ of $Y$
such that $U$ is contained in a leaf of $\F$.

\subsection{Differential forms defining foliations}
If $\F$ is a foliation of codimension $q$ then the $q$-th exterior power of the inclusion  $N^*_{\F} \to \Omega^1_X$ twisted by the line-bundle $\mathcal N = \det(N_{\F}^*)^*$ gives rise to a twisted $q$-form  $\omega \in H^0(X, \Omega^q_X \otimes \mathcal N)$ which is locally decomposable (at a general point $x \in X$, $\omega$ can be locally expressed as the product $\omega_1 \wedge \cdots \wedge \omega_q$ of $q$ distinct $1$-forms) and integrable (the $1$-forms $\omega_i$ satisfy $\omega_i \wedge d \omega =0$). Moreover, the tangent sheaf of $\F$ can be recovered as the kernel of the morphism
\[
    T_X \longrightarrow \Omega^{q-1}_X \otimes \mathcal N
\]
defined by contraction with $\omega$.

Reciprocally, if $\omega \in H^0(X,\Omega^q_X \otimes \mathcal N)$ is a locally decomposable and integrable $q$-form then the foliation defined by $\omega$ is, by definition, the foliation $\F_{\omega}$ with tangent sheaf $T_{\F_{\omega}}$  equal to the kernel of the morphism from $T_X$ to $\Omega^{q-1}_X \otimes \mathcal N$ defined by contraction with $\omega$. In general the line-bundle $\mathcal N$ does not coincide with the line-bundle $(\det N^*_{\F_{\omega}})^*$. This only happens when the zero locus of $\omega$ has codimension at least two.

\subsection{Symmetries of foliations}
Let $f : X \dashrightarrow Y$ be a dominant meromorphic map between complex manifolds. If $\G$ is a codimension $q$ foliation on $Y$ defined by a twisted $q$-form $\omega \in H^0(Y, \Omega^q_Y \otimes \mathcal L)$ then $f^* \G$ is the codimension $q$ foliation on $X$ defined by $f^* \omega \in H^0(X, \Omega^q_X \otimes f^* \mathcal L)$.

For a foliation $\F$ on a complex manifold $X$ we will denote by $\Aut(X,\F) \subset \Aut(X)$ the subgroup of the group of automorphisms of $X$ formed by automorphisms $f: X \to X $ such that $f^* \F = \F$. Similarly, we will denote by $\Bir(X,\F) \subset \Bir(X)$ the subgroup of the group of bimeromorphic transformations $f: X \dashrightarrow X$ of $X$ such that $f^* \F = \F$. We chose to use $\Bir(X,\F)$ instead of $\Bim(X,\F)$ since our focus in this paper will be on bimeromorphisms of foliations on projective manifolds, which are birational maps thanks to Chow Theorem.

We are also interested in the semigroup (actually the monoid) $\End(X,\F)$ of meromorphic/rational endomorphisms of a foliation $\F$ which consists of dominant meromorphic maps $f: X \dashrightarrow X$ such that $f^* \F = \F$. Of course, $\Aut(X,\F) \subset \Bir(X,\F) \subset \End(X,\F)$  and it is not hard to produce examples where all the inclusions are strict.

If $f \in \Aut(X,\F), \Bir(X,\F)$, or $\End(X,\F)$ then $f$ acts on the set of leaves of $\F$. When $X$ is compact, we define $\AutFix(X,\F) \subset \Aut(X,\F)$, $\BirFix(X,\F) \subset \Bir(X,\F)$, and $\EndFix(X,\F) \subset \End(X,\F)$ as the sub(semi)groups which act trivially on the set of leaves. More precisely, $f \in \AutFix(X,\F), \BirFix(X,\F)$, or $\EndFix(X,\F)$ if, and only if,  for any general $x \in X$ the leaves of $\F$ through $x$ and $f(x)$ coincide.
Note that $\AutFix(X,\F)$ and $\BirFix(X,\F)$ are, respectively, normal subgroups of $\Aut(X,\F)$ and $\Bir(X,\F)$.  but  $\EndFix(X,\F)$ may fail to be normal (in the semigroup sense) in  $\EndFix(X,\F)$. However, it is still possible to give a reasonable sense to the quotient as a semigroup :
\begin{dfn}
	Consider on  $\End(X,\F)$ the equivalence relation $\sim$ defined by $u\sim v$ if for a general point $x$,
    the leaves through $u(x)$ and $v(x)$ coincide. This is clearly a semigroup congruence; moreover ${[\Id]}_\sim={\EndFix(X,\F)}$.
    One then sets	
    $$
        \frac{\End(X,\F)}{\EndFix(X,\F)}:=\frac{\End(X,\F)}{\sim}.
    $$
\end{dfn}

The following proposition is well-known to specialists, cf. \cite[Corollary 2]{MR1957664}.

\begin{prop}\label{P:transverse symmetries}
Let $\F$ be a codimension one foliation on a complex manifold $X$.
Suppose that $\F$ is invariant by the flow of a vector field $v\in H^0(X,T_X)$  which is not everywhere tangent to $\F$.
Then $\F$ is defined by a closed meromorphic $1$-form.
\end{prop}
\begin{proof}
    Let $\omega \in H^0(X,\Omega^1_X \otimes \mathcal N)$ be a twisted $1$-form on $X$ defining $\F$. By assumption $\omega(v) \in H^0(X, \mathcal N)$ is non-zero. Therefore the quotient $\tilde \omega:=\omega/\omega(v)$ is a meromorphic $1$-form.

    To prove the proposition it suffices to verify that the meromorphic form $\tilde \omega$ is closed. It is enough to check that $d\tilde \omega=0$ at smooth points of $\F$ such that $v$ is locally transverse to $\F$. In a neighborhood of such a point we can find local coordinates $(z,w_1,\ldots ,w_n)=(z,w)$ such that
    \[
        \omega=a(z,w)\, dz, \qquad v=b(z,w) \frac{\partial}{\partial z}.
    \]
    The condition that the flow of $v$ preserves $\F$ means that $b(z,w)=b(z)$ does not actually depend on $w$. Therefore $\tilde w =  dz/b(z)$ is closed as claimed.
\end{proof}

\subsection{Rational endomorphisms of foliations on projective surfaces}
As already mentioned in the introduction, the groups of birational transformations of foliations on  projective surfaces were extensively studied in \cite{MR1998612}.  There, the authors classify foliations on surfaces with infinite group of birational transformations. As a corollary, they deduce that foliations with and infinite group of birational transformations are Liouvillian integrable, see \cite[Corollary 7.3]{MR1998612}. Their proof shows that for foliations with $\Bir(S,\F)$ infinite, there exists a generically finite morphism $\pi:S' \to S$ such that $\pi^*\mc F$ is defined by a closed rational $1$-form. The proof of  \cite[Theorem A]{MR2818727} shows that the very same statement holds for the semigroup of rational endomorphisms of a foliation on a projective surface.

\begin{thm}\label{T:classification surfaces}
    Let $S$ be a projective surface and $\mc F$ a foliation on $S$.  If $\End(S, \F)$ is infinite then there exists a generically finite morphism $\pi:S' \to S$ such that $\pi^*\mc F$ is defined by a closed rational $1$-form.
\end{thm}

\subsection{Algebraic and transcendental parts of a foliation}\label{SS:maximal algebraic}
Let $\F$ be a foliation on a projective manifold $X$. According to \cite[Lemma 2.4]{MR3842065}, given a foliation $\F$ on a projective manifold, there exists a unique foliation $\F^{\alg}$ by algebraic leaves characterized by the following property: for a very general point $x \in X$ the germ of the leaf of $\F^{\alg}$ through $x$ coincides with the maximal germ of algebraic subvariety contained in the leaf of $\F$ through $x$. In the terminology of \cite{MR3652251}, $\F^{\alg}$ is the algebraic part of the foliation $\F$.

It also follows from \cite[Lemma 2.4]{MR3842065}, the existence of a projective manifold $Y$ (unique up to birational transformations), a dominant rational map $\pi: X \dashrightarrow Y$ with connected fibers, and a foliation $\F^{\trans}$ with trivial algebraic part  on $Y$ such that $\F = \pi^* \F^{\trans}$.  In the terminology of \cite{MR3652251}, $\F^{\trans}$ is the \emph{transcendental part} of the foliation $\F$ and the dimension of $\F^{tr}$ is called the \emph{transcendental dimension} of $\F$. We will call $\pi : X \rato Y$ the \emph{maximal algebraic quotient} of $\F$. Note that $\F^{alg}$ is nothing but the foliation on $X$ defined by $\pi$. A foliation with trivial algebraic part will be called \emph{purely transcendental}.

The birational equivalence classes of the variety $Y$ and of the rational map $\pi$ are characterized by the identification of the field $\mathbb C(\F^{\alg})$  of rational first integrals of $\F^{\alg}$ with the pull-back under $\pi$ of the field $\C(Y)$ of rational functions on $Y$.

\begin{prop}\label{P:reductiontotranscendental}
    If $\F$ is a foliation on a projective manifold $X$ then there exists a natural morphism of semigroups
    \[
        \End(X,\F) \longrightarrow \End(Y,\F^{\trans})
    \]
    which maps $\EndFix(X,\F)$ to $\EndFix(Y,\F^{\trans})$. Moreover, the induced morphism
    \[
         \frac{\End(X,\F)}{\EndFix(X,\F)} \longrightarrow \frac{\End(Y,\F^{\trans})}{\EndFix(Y,\F^{\trans})}
    \]
    is injective. In particular, if the action of $\End(Y,\F^{\trans})$ is transversely finite then the action of $\End(X,\F)$ is transversely
    finite.
\end{prop}
\begin{proof}
    Let $\pi: X \rato Y$ be the maximal algebraic quotient of $\F$. Fix $f\in \End(X,\F)$ and suppose that for a general fibre $F$ of $\pi$ we have $\pi(f(F))\neq \{\text{pt.}\}$. Since $F - \sing(\F)$ is contained in a leaf of $\F$ and $f$ preserves the foliation $\F$, $f(F)$ is also tangent to $\F$.
    Thus $\pi(f(F))$ is tangent to $\F^{\trans}$. But this is impossible since $\F^{\trans}$ is the transcendental part of $\F$. It follows that $f \in \End(X,\F)$ induces a rational transformation of $Y$ preserving $\F^{\trans}$. Clearly, if $f \in \EndFix(X,\F)$ then the induced transformation belongs to $\EndFix(Y,\F^{\trans})$. In this way, we obtain a natural commutative diagram
    \[
    \begin{tikzcd}
        \End(X,\F) \arrow[r] \arrow[d] &  \End(Y,\F^{\trans}) \arrow[d] \\
        \displaystyle{\frac{\End(X,\F)}{\EndFix(X,\F)}} \arrow[r] &  \displaystyle{\frac{\End(Y,\F^{\trans})}{\EndFix(Y,\F^{\trans})}}
    \end{tikzcd}
    \]
    of semigroup morphisms.  Since the kernel of the top arrow is contained in $\EndFix(X,\F)$, it follows that the bottom arrow is
    injective.
\end{proof}

\begin{cor}\label{C:reductiontranscendental}
    Let $\F$ be a codimension one foliation on a projective manifold $X$ and let $\pi\colon X\rato Y$ be the maximal algebraic quotient of $\F$.
    If the action of $\End(X,\F)$ is not transversely finite and the transcendental dimension of $\F$ is one (i.e., $\dim Y=2$) then there exists a generically finite morphism $\pi:Y' \to Y$ such that $\pi^*\mc F$ is  defined by a closed rational $1$-form.
\end{cor}
\begin{proof}
    The result follows from Theorem \ref{T:classification surfaces} combined with Proposition \ref{P:reductiontotranscendental}.
\end{proof}

\section{Transversely homogeneous codimension one foliations}\label{S:Transversely}
From a dynamical point of view, the simplest possible foliations on projective manifolds are the ones with all leaves algebraic. It is well-known that the general leaves of  one such foliation are defined by the level sets of a rational map, see \cite{MR1017286}. Because of that, foliations with all leaves algebraic are called algebraically integrable foliations.

For foliations of codimension one, there is a well-established hierarchy of generalizations of the concept of algebraically integrable foliations. In order of growing complexity: the transversely additive, transversely affine, and transversely projective foliations.

\subsection{Transversely additive foliations}\label{SS:transv add}
The definition of a transverse additive structure for a codimension one foliation is relatively simple.

\begin{dfn}\label{def additive}
    Let $\F$ be a codimension one foliation on a complex manifold $X$. A \emph{transverse additive structure} for $\F$ is a closed meromorphic
    $1$-form $\omega$ such that $\F= \F_{\omega}$.
\end{dfn}

Two transverse additive structures $\omega_1, \omega_2$ for $\F$ are said to be equivalent if there exists a constant $\lambda \in \mathbb C^*$ such that $\omega_1 = \lambda \cdot \omega_2$. We will say that a foliation $\F$ is transversely additive if there exists an transversely additive structure for $\F$.

A slightly more general class is the class of virtually transversely additive foliations.  A foliation is virtually transversely additive if there exists a generically finite dominant meromorphic map $\pi: Y \rato X$ such that $\pi^* \F$ is transversely additive.

\subsection{Transversely affine foliations}\label{SS:transversely affine}
For a thorough discussion of the definition transverse affine structure for a codimension one foliation presented below,  see \cite{MR3294560}.

\begin{dfn}\label{def affine}
    Let $\F$ be a codimension one foliation on a complex manifold $X$ defined by a twisted $1$-form $\omega \in H^0(X, \Omega^1_X \otimes N_{\F})$ with
    singularities of codimension at least two. A \emph{transverse affine structure} for $\F$ is a flat meromorphic connection $\nabla$ on $N_{\F}$ such that $\nabla(\omega)=0$.
\end{dfn}

If $\F$ admits two distinct transverse affine structures $\nabla_1, \nabla_2$ then they differ by a closed meromorphic $1$-form defining $\F$. Thus, their difference is a transverse additive structure for $\F$.  A foliation is transversely affine if it admits a transverse affine structure.

Every virtually transversely additive foliation is also a transversely affine foliation. In this case, the flat meromorphic connection $\nabla$ on $N_{\F}$ is a logarithmic connection with finite monodromy, see \cite[Section 2.1]{2017arXiv171208330L}.

\subsection{Transversely projective foliations}\label{SS:transv proj}
We recall below the definition of a (singular or meromorphic) transverse projective structure for codimension one singular holomorphic foliations presented in  \cite{MR2337401, MR3522824}.  For alternative definitions, see references therein. It is a generalization of the classical definition of transverse projective structures for codimension one smooth foliations presented in \cite{MR1120547}.

\begin{dfn}\label{def projective}
    Let $\F$ be a codimension one foliation on a complex manifold $X$. A \emph{transverse projective structure} for $\F$ is the data of a triple $(E, \nabla, \sigma)$ where
    \begin{enumerate}
        \item $E$ is a rank $2$ vector bundle;
        \item $\nabla$ is a flat meromorphic connection on $E$;
        \item $\sigma \colon X \rato \mathbb P(E)$ is a meromorphic section of $\Pj(E)\to X$ such that, if $\mc R$
        denotes the Riccati foliation on $\Pj(E)$ with general leaf given by  projectivization of a flat section of $\nabla$, $\F=\sigma^*\mc R$.
    \end{enumerate}
\end{dfn}

Two transverse projective structures  $(E_1, \nabla_1, \sigma_1)$ and $(E_2,\nabla_2, \sigma_2)$ for a foliation $\F$ are equivalent (see \cite{MR3522824}) if there exists a bimeromorphic bundle transformation $\Phi : \mathbb P(E_1) \rato \mathbb P(E_2)$ such that $\Phi^*(\mc R_2) = \mc R_1$ and which makes  the following diagram
\[
\begin{tikzcd}
    \mathbb P (E_1)  \arrow[dr] \arrow[rr,dashed, "\Phi"] & & \mathbb P(E_2) \arrow[dl] \\
    & \arrow[ul,bend left, dashed, "\sigma_1"] X \arrow[ur,bend right,dashed, "\sigma_2" right]&
\end{tikzcd}
\]
commutative.

We will say that a foliation is transversely projective if it admits a transverse projective structure. When $X$ is a  projective manifold, as a consequence of GAGA's principle, such a projective structure can be defined up to equivalence by a triple $(E,\nabla,\sigma)$ where $E$ is the \textit{trivial} rank $2$ vector bundle over $X$.

It can be verified that every transversely affine foliation is also a transversely projective foliation, see \cite{MR3294560}.

As in the case of transversely additive and transversely affine foliations, the existence of two non-equivalent projective structures for $\F$ implies that $\F$ admits a more restrictive transverse structure as stated in the following lemma.

\begin{lemma} \label{uniqueness projective structure}
    Suppose $\F$ is a codimension one foliation on a projective manifold $X$ that admits more than one transverse projective structure. Then, there exists $\pi \colon X'\to X$, where $\pi$ is  generically finite of topological degree at most two, such that $\pi^* \F$ is defined by a closed rational $1$-form. In particular, $\F$ is virtually transversely additive.
\end{lemma}
\begin{proof}
     Follows from the proof of \cite[Lemma 2.20]{MR2324555}.
\end{proof}

If $X$ is a projective manifold carrying a transversely projective foliation $\F$ and $\pi:X \dashrightarrow Y$ is a dominant rational map between projective manifolds, the foliation $\pi^* \F$ is also transversely projective. Just exploit the fact that the underlying bundle $E$ attached to the given projective structure can be chosen to be trivial. It is worth noticing that the converse also holds:

\begin{lemma}\label{L:CasaleCRAS}
	Let $\pi : X \dashrightarrow Y$ be a dominant rational map between projective manifolds. If $\mathcal F$ is a codimension one
	foliation on $Y$ such that $\pi^* \mathcal F$ is transversely projective
	then $\mathcal F$ itself is also transversely projective.
\end{lemma}
\begin{proof}
	This is \cite[Lemma 3.2]{MR1955577}, see also \cite[Lemma 2.1]{2017arXiv171208330L}.
\end{proof}

\subsection{Distinguished first integrals and monodromy representation}\label{SS:distinguished}
Fix a transverse projective structure for a transversely projective foliation $\F$. Let $H \subset X$ denote the polar hypersurface of the connection $\nabla$. Locally at points of $X_0:=X\setminus H$ the foliation $\F$ admits distinguished (meromorphic) first integrals
\[
    F_i\colon U_i \rato \Pj^1
\]
which are uniquely defined modulo composition to the left by elements of $\Aut(\Pj^1)=\PSL_2(\C)$.

If we fix $x_0 \in X_0$, the analytic continuation of one such distinguished local first integral along closed paths yields a (meromorphic) developing map
\[
    \dev \colon \widetilde X_0 \rato \mb \Pj^1\, ,
\]
where $\widetilde X_0$ denotes the universal cover of $X_0$, and a monodromy representation
\[
    \rho\colon \pi_1(X_0,x_0)\to \PSL_2(\C)
\]
such that
\[
    \rho(\gamma) \circ\dev=\dev \circ \gamma \, , \qquad \text{for all }\gamma\in \pi_1(X_0).
\]

Neither the developing map nor the monodromy representation is uniquely determined, even after
fixing the base point $x_0$. They depend on the choice of the initial local first integral. A different choice of a distinguished local first integral alters the developing mapping by left composition with an automorphism $\varphi$ of $\mathbb P^1$. After the alteration, the monodromy representation becomes $\gamma \mapsto \varphi \circ \rho(\gamma) \circ \varphi^{-1}$.

\begin{lemma}\label{L:uniqueness}
    Let $\F$  be a transversely projective foliation on a projective manifold $X$ which is not algebraically integrable. If its monodromy representation $\rho:\pi_1(X_0)\to \PSL_2(\C)$ has Zariski dense image then $\F$ admits exactly one transverse projective structure.
\end{lemma}
\begin{proof}
    One proceeds by contradiction, assuming the non-uniqueness of transverse projective structure. By  Lemma \ref{uniqueness projective structure}, it follows that, up to taking pull-back by a generically finite map (what does not affect the denseness of $\rho$, nor the non algebraic integrability), $\F$ is also defined by a closed rational one form $\Omega$.  Then, by taking the Schwartzian derivative $\{F,f \}$ of a local distinguished first integral $F$ associated to the original structure with respect to a distinguished first integral $f$ associated to the second structure(i.e., $f$ a local primitive of $\Omega$), and according to the Schwartzian derivative rule, one inherits a rational quadratic differential
    \[
        \Theta= G(\Omega\otimes\Omega)
    \]
    where $G=\{ F,f\}$ (locally defined) is a  rational first integral for $\F$ (globally defined). Because $F\not=f$ (modulo left composition by automorphisms of $\Pj^1$) $G$ does not vanish identically. Moreover $G$ can't be a non zero constant, otherwise there would exist $\lambda\in{ \C}^*$ such that $F=e^{\lambda f}$ (modulo $\Aut{( \Pj^1)}$). This would imply that the image of $\rho$ fixes a point of $\Pj^1$, thus contradicting the assumption of Zariski denseness. Then  the rational function $G$ is necessarily non constant.
    This contradicts the hypothesis that $\F$ is not algebraically integrable.
\end{proof}

\subsection{Singularities of transverse projective structures}\label{S:singtransvproj}
Transverse projective structures inherit the concept of regular and irregular singularities from flat meromorphic connections.

\begin{dfn}
    We say that a transversely projective foliation has regular singularities if there exists a representative $(E, \nabla, \sigma)$ of the transverse projective structure for $\F$ such that the connection $\nabla$ has at worst regular singularities in the sense of \cite[Chapter 2, Definition 4.2]{MR0417174}.
\end{dfn}

If a transverse projective structure has regular singularities then, as shown in \cite[Lemma 2.4]{MR3294560}, the monodromy representation determines, up to birational gauge equivalence, the Riccati foliation on $\mathbb P(E)$ defined by $\nabla$. For later use, let us state a direct consequence of this fact.

\begin{prop}\label{abelian monodromy}
    Let $\F$ be a codimension one foliation on a complex manifold endowed with a transverse projective structure having at worst regular singularities. Then the following assertions hold true.
    \begin{enumerate}
        \item If the monodromy representation is finite then $\F$ admits a meromorphic first integral.
        \item If the monodromy representation is virtually abelian then $\F$ is (virtually) transversely additive.
        \item If the monodromy representation is solvable then $\F$ is transversely affine.
    \end{enumerate}
\end{prop}

In particular, if a transverse projective structure has regular singularities and the monodromy (of a small loop) around an irreducible hypersurface $D\subset H$ is trivial, then a distinguished first integral defined in a neighborhood of $\gamma$ extends meromorphically through $D$.

\begin{rem}
    Beware that algebraically integrable foliations admit infinitely many non-equivalent transverse projective structures. In particular, they admit
    transverse projective structures with monodromy Zariski dense in $\PSL(2,\C)$. Also, there exist transversely projective foliations $\F$ which
    are not transversely affine but have trivial monodromy. Of course, if this is the case, then $\F$ must have irregular singularities.
\end{rem}

\subsection{Transversely hyperbolic foliations}\label{SS:transv hyp}
We keep the notation used in Subsection \ref{SS:distinguished}. A transversely hyperbolic foliation is a transversely projective foliation with a developing map $\dev \colon \widetilde X_0 \rato \mb \Pj^1$ having image contained in the unitary disc $\mathbb D \subset \C \subset \mathbb P^1$ and a monodromy representation $\rho$ taking values in  $\Aut(\mathbb D) \subset \PSL_2(\mathbb C)$.

Although both conditions are somewhat restrictive, transversely hyperbolic foliations are central to the theory, as is shown by the structure theorem for transversely projective foliations  presented in Subsection \ref{SS:structure}.

\section{Transverse symmetries and orbit closures}\label{S:orbits}

Let $f: X \dashrightarrow X$ be a dominant rational map of a projective manifold. We will say that the orbit dimension of $f$ is equal to $d$ if, for any sufficiently general point $x \in X$, the Zariski closure of the $f$-orbit $\{ f^n(x) ; n \in \mathbb N\}$ of $x$ has dimension $d$. Similarly, if $G \subset \End(X)$ is a semi-group of dominant rational maps $X$, we define the orbit dimension of $G$ as the dimension of the Zariski closure of the $G$-orbit of a very general point $x \in X$.

This section is devoted to the proof of the following result.

\begin{prop}\label{P:orbit12}
    Let $X$ be projective manifold and let $\mathcal F$ be a purely transcendental codimension one foliation on $X$. If $\End(X,\F)$ has orbit dimension one or two then $\mathcal F$ is a virtually transversely additive foliation.
\end{prop}

\subsection{Results used in the proof of Proposition \ref{P:orbit12}}\label{S:factorization results}
An important tool used in the proof of Proposition \ref{P:orbit12} is the following  result due Bell, Ghioca, and Reichstein

\begin{thm}{\cite[Theorem 4.1]{MR3676045}}\label{T:Belletal}
    Let $X$ be a projective manifold and let $G \subset \End(X)$ be a semigroup of dominant rational maps. Then there exists a rational  map $ f : X \dashrightarrow T$ to a projective manifold $T$ such that
    \begin{enumerate}
        \item the map $f$ is $G$-invariant, i.e. $g\circ f = g$ for any $g \in G$; and
        \item the Zariski closure of the fiber of $f$ containing a sufficiently general point of $x \in X$ coincides with the Zariski closure of the $G$-orbit of $x$, i.e.
            \[
                \overline{f^{-1}(f(x))} = \overline{ \left\{ g(x) ; g \in G \right\} } \, .
            \]
    \end{enumerate}
\end{thm}

Note that the fibers of $f$ are not necessarily irreducible but, after replacing $G$ by a finite index sub semigroup, one can assume that this is the case.

When $G$ is the semigroup generated by one dominant rational map, the result above is due to Amerik and Campana, see  \cite[Theorem 4.1]{MR2400885}, and was originally stated in a more general setup (compact Kähler manifolds).

We will say that an orbit of a semigroup $G \subset \End(X)$   is very general if it corresponds to a very general point of the variety $T$ provided by  Theorem \ref{T:Belletal}.

We will also make use of the following result proved in \cite{2017arXiv171208330L}.

\begin{thm}\label{T:extensionoftransversestructure}.
    Let $\mathcal F$ be a codimension one foliation on a projective manifold $X$. Suppose there exist a projective variety $B$ and a morphism $g : X \to B$ with irreducible general fiber  such that $\F$ is transversely projective/transversely affine/virtually transversely additive when restricted to a very general fiber of $g$. If the restriction of $\F$ to the very general fiber does not admit a rational first integral, then $\F$ is, respectively, transversely projective/transversely affine/virtually transversely additive.
\end{thm}

\subsection{Orbit  dimension two} We will split the proof of Proposition \ref{P:orbit12} in two parts according to the orbit dimension of $\End(X,\F)$. The lemma below settles the case of orbit dimension two.

\begin{lemma}\label{L:reduz}
    Let $\F$ be a codimension one foliation on a projective manifold $X$ and let $G \subset \End(X,\F)$ be a sub-semigroup. Assume that $\mathcal F$ is purely transcendental and the orbit dimension of $G$ is at least two. Suppose in addition that  the restriction of $\mathcal F$ to the irreducible components of the Zariski closure of a very general orbit of $G$ is virtually transversely additive, transversely affine, or transversely projective, then $\F$ itself is virtually transversely additive, transversely affine, or transversely projective respectively.
\end{lemma}
\begin{proof}
    Replace $G$ by a finite index subsemigroup so that the Zariski closure of a general orbit is irreducible and apply Theorem \ref{T:Belletal} combined with Theorem \ref{T:extensionoftransversestructure}. The hypothesis on the non-existence of first integrals for the restriction of $\F$ to the very general fiber of $G$ in Theorem \ref{T:extensionoftransversestructure} is satisfied because the algebraic part of $\F$ is trivial.
\end{proof}

\begin{cor}\label{C:orbit 2}
    Let $\F$ be a codimension one foliation on a projective manifold $X$.  If $\mathcal F$ is purely transcendental and there exists a sub-semigroup $G \subset \End(X,\F)$ with orbit dimension equal to two then $\F$ is a virtually transversely additive foliation.
\end{cor}
\begin{proof}
    Since the orbit dimension of $G$ is exactly two, we can apply Theorem \ref{T:classification surfaces} to deduce that the restriction of $\F$ to any  irreducible component of the Zariski closure of a general orbit of $f$ is a virtually transversely additive foliation. We apply Lemma \ref{L:reduz} to conclude. 
\end{proof}

\subsection{Orbit dimension one}
We will now study semigroups contained in $\End(X,\F)$ with orbit dimension equal to one. We start with semigroups
generated by a single rational map.

\begin{prop}\label{P:dim orbit 1}
    Let $\F$ be a codimension one foliation on a projective manifold $X$ invariant by a rational map $g: X \dashrightarrow X$. If $\mathcal F$ is purely transcendental and the orbit dimension of $g$ is equal to one then $\F$ is a virtually transversely additive foliation.
\end{prop}
\begin{proof}
    Apply Theorem \ref{T:Belletal} to obtain a dominant rational map $f:X \dashrightarrow T$ such that $f\circ g = f$ and $f$ has fiber dimension one.    There is no loss of generality in assuming that the general fiber of $f$ is irreducible (after replacing $g$ by a suitable power) and that $f$ is a morphism (resolve indeterminacies of $f$).

    Since the orbit dimension of $g$ is one, $g$ has infinite order and the same holds true for the restriction of $g$ to a general fiber of $f$. Therefore the genus of a general fiber of $f$ is equal to zero or one.

    First assume that the general fiber of $f$ has genus one and let $E$ be a general fiber. The locus of tangencies between the fiber $E$ and the foliation $\F$ is clearly both forward and backward invariant by $g$. Since infinite order morphisms of genus one curves have no finite backward orbits and the algebraic part of $\F$ is trivial, we deduce that $\F$ is completely transverse to the general fiber of $f$. The transversality between $\F$ and the general fiber of $f$ implies that the family of genus one curves defined by $f$ is isotrivial. Hence, there exists a generically finite rational map $\beta : B \to T$ such that the (main component of the normalization of the) fiber product $X\times_T B$ is birationally equivalent to the product $B \times E$.  The foliation $\F$ lifts to a foliation $\tilde \F$ defined on $B\times E$ by a rational $1$-form $\Omega$ that can be written as
    \[
      \Omega =   d z + \sum_{i=1}^{\dim B} \omega_i \cdot a_i(x,z)
    \]
    where $\omega_i$ are (pull-backs of) rational $1$-forms on $B$ and $a_i \in \mathbb C(B\times E)$ are rational functions.
    The transversality of $\F$ with the general fiber of $f$ implies that the functions $a_i$ do not depend on the variable $z$. We can thus write
    \[
        \Omega = dz  + \omega
    \]
    where $\omega$ is a rational $1$-form on $B$.  Frobenius integrability condition implies that $d \omega=0$, showing that $\tilde \F$ is a transversely additive foliation. We can apply a well-known lemma, see for instance \cite[Lemma 2.1]{2017arXiv171208330L}, to conclude that $\F$ is virtually transversely additive.

    Let us now treat the case where the general fiber of $f$ has genus zero.  Let $\Sigma \subset X$ be the set of fixed points of $g$. Since rational maps of $\mathbb P^1$ distinct from the identity have one or two finite backward orbits, there are one or two irreducible components of $\Sigma$ which dominate $T$. We can assume (perhaps after applying a generically $2:1$ base change) that each one of these irreducible components is birational to the base $T$.
    Moreover, after replacing $g$ by $g^2$ if needed, we can assume that the point(s) with finite backward orbit are also totally invariant fixed points of $g$. As already argued in the proof of the genus one case, it suffices to show that the foliation is virtually transversely additive under this extra assumption.

    Choose a birational model of $X$ isomorphic to $B \times \mathbb P^1$ where $g$ can be written as $g(x,z) = (x, p(x,z))$ where $p \in \mathbb C(X)[z]$
    is a polynomial on $z$ of degree $d$ with coefficients on the field of rational functions of $X$. In other words, we are assuming that one of the totally invariant fixed points of $f$ is located at $\{ z = \infty\}$. If there exists a second totally invariant fixed point then assume that it is located at $\{ z = 0\}$. With this normalization, and since the tangencies between $\F$ and the general fiber of $f$ are contained in
    $\{ z = \infty\} \cup \{ z=0\}$, we have that the foliation $\F$ is defined by a rational $1$-form $\omega$ that can be written as
    \[
        \omega = dz + \sum_{i=\ell}^u \omega_i z^i \, , \text{ with } \omega_{\ell} \neq 0 \text{ and } \omega_u \neq 0 \, ,
    \]
    where $\omega_i$ are rational $1$-forms on $B$, the lower limit $\ell$ is negative if and only if $\F$ is tangent to $f$  at $\{ z=0\}$, and the upper limit $u$ is strictly greater than $2$ if and only if $\F$ is tangent to $f$ at $\{z = \infty\}$.

    Since
    \[
        g^* \omega = d p + \sum_{i=\ell}^{u} \omega_i p^i \, ,
    \]
    the $g$-invariance of $\F$  implies that
    \begin{equation}\label{E:boba}
        g^* \omega =  \frac{\partial p}{\partial z}  \cdot \omega  \, .
    \end{equation}
    Compare the coefficients of highest $z$-degree in the Equation (\ref{E:boba}).
    If we assume that the upper limit $u$ is different from zero then
    we deduce that
    \[
        d \cdot u = (d-1) + u .
    \]
    If $u\ge 2$ then $d=1$ and  we can write $p(x,z) = a(x) z + b(x)$. Comparing the coefficients
    of $z^u$ we get that $\omega_u a(x)^u = a(x) \cdot \omega_u$. Thus $u\ge 2$ cannot happen and we
    deduce that $u \in \{ 0,1\}$. Similarly, if  we assume that the lower bound $\ell$ is negative and we compare the coefficients of lowest $z$-degree in Equation (\ref{E:boba}), we reach a contradiction and deduce that $\ell$ also belongs to $\{ 0 , 1\}$.

    At this point, we can write
    \[
        \omega = dz + \omega_0 + \omega_1 \cdot z
    \]
    what is sufficient to deduce that $\F$ is transversely affine. One can proceed with an elementary analysis, similar
    to what have been so far, to further restrict the transverse structure of $\F$ showing that it is in fact virtually
    transversely additive. Here we will adopt an alternative approach. Theorem \ref{THM:D} implies that either $\F$ is virtually transversely Euclidean as wanted, or $\F$ is a foliation on a surface.  We can apply Theorem  \ref{T:classification surfaces} to conclude.
\end{proof}

The proposition below extends the one above to arbitrary semigroups.

\begin{prop}\label{P:dim orbit 1 bis}
    Let $\F$ be a codimension one foliation on a projective manifold $X$ and let $G \subset \End(X,\F)$ be a subsemigroup.
    If $\mathcal F$ is purely transcendental and the orbit dimension of $G$ is equal to one then $\F$ is a virtually transversely additive foliation.
\end{prop}
\begin{proof}
    If $G$ has an element of infinite order the result follows from Proposition \ref{P:dim orbit 1}. Assume that
    every element of $G$ is of finite order. In particular, $G \subset \Bir(X,\F)$ is a group. Let  $f : X \dashrightarrow T$ be the
    rational map given by Theorem \ref{T:Belletal} applied to $G$. As before, we are free to assume that the general fiber
    of $f$ is irreducible.

    If the general fiber of $f$ is an elliptic curve then the Zariski denseness of the orbits of $G$ on a
    general fiber is sufficient to guarantee that $\F$ is everywhere transverse to a general fiber
    and we can conclude as in the proof of Proposition \ref{P:dim orbit 1}.

    If the general fiber of $f$ is $\mathbb P^1$ then, after replacing $G$ by a  sub-semigroup of index two, we
    can assume that the action of any element of $G$ at a general fiber of $f$ has exactly two fixed points. This
    implies, perhaps after taking a degree two base change, that $X$ is birationally equivalent to $T \times \mathbb P^1$ and $G$ is conjugated to a subgroup of $\mathbb C^*$ acting on $T \times \mathbb P^1$ by morphisms of the form $(x,z) \mapsto (x , \lambda z)$, $ \lambda \in \mathbb C^*$.
    In these coordinates it is clear that $\Bir(X,\F)$ contains the flow of a vector field and we can apply Proposition \ref{P:transverse symmetries} to conclude.
\end{proof}

\begin{cor}\label{C:reduz}
    Let $\F$ be a purely transcendental codimension one foliation on a projective manifold $X$,  let $G \subset \End(X,F)$ be a
    semigroup of positive orbit dimension, and let  $\G$ be  the restriction of $\F$ to an irreducible component of the Zariski closure of a very general orbit of $G$. If $\G$ is transversely projective then $\F$ itself  is also transversely projective.
\end{cor}
\begin{proof}
    If the orbit dimension of $G$ is one then $\G$ is a foliation by points and is automatically transversely projective. We can apply Proposition \ref{P:dim orbit 1 bis} to deduce that  $\F$ is virtually transversely additive and, in particular, $\F$ is transversely projective.

    If the orbit dimension of $G$ is at least two, then the result follows from Lemma \ref{L:reduz}.
\end{proof}

\subsection{Proof of Proposition \ref{P:orbit12}}
It suffices to combine  Corollary \ref{C:orbit 2} with Proposition \ref{P:dim orbit 1 bis}.
\qed

\section{Transversely hyperbolic foliations}\label{S:transv hyp}

\subsection{Conventions}
Recall from \S \ref{SS:transv hyp} that a transversely hyperbolic foliation on a complex manifold $X$ is a transversely projective foliation that
has one of its developing maps taking values in the Poincaré disc $\mathbb D$. Local determinations of the developing map are not everywhere defined, but \textit{a priori} only on the complement of the polar divisor of the underlying projective structure.

We will say that a transversely projective/hyperbolic structure $(E,\nabla, \sigma)$ for a foliation $\F$ on a complex manifold $X$ is without poles
if the Riccati foliation determined by $\nabla$ is everywhere transverse to the fibration $\mathbb P(E) \to X$. In this case, the developing map of the transverse structure is everywhere locally defined.  A foliation admitting a transversely projective/hyperbolic structure without poles will be called a transversely projective/hyperbolic foliation without poles.

\subsection{Extension through simple normal crossing divisors}
It is convenient to interpret a  transversely hyperbolic foliation without poles $\F$ on a complex manifold $X_0$ as a $\rho$-equivariant (non constant) map $\mathcal D:\tilde{X_0} \to \D$, where $\rho:\pi_1(X_0,x_0)\to \Aut(\D)$ is a morphism and $\tilde{X_0}$ is the universal covering of $X_0$. Actually, $\mathcal D$ is a developing map attached to the underlying projective structure and is unique modulo left composition by $\Aut{( \D)}$. We will use the notation $f=\mathcal D \circ \varphi$ where $\varphi$ is a local or multivalued inverse of the covering map $\pi:\tilde{X_0}\to X_0$. One must think of $f$ as a multivalued map on $X_0$ with monodromy representation $\rho$. In particular, if  $Y$ is a connected open subset or subvariety of $X_0$,  it makes sense to consider the restriction of $f$ to $Y$ by considering its restriction to a connected component of $\pi^{-1}(Y)$.

Assume $X_0$ is the complement of a simple normal crossing hypersurface on a complex  manifold $X$. Our goal in this subsection is to show that a transversely hyperbolic foliation without poles defined on $X_0$ %$, the complement of a simple normal crossing divisor on a complex manifold $X$,
extends as a foliation to the whole complex manifold $X$. Furthermore, we will show that the transversely hyperbolic structure also extends to the whole $X$, but perhaps will acquire poles along the boundary divisor.

\begin{lemma}\label{L:Rextension}
    Notations as above. Assume that $X_0=X-H$, where $H\subset X$ is a hypersurface on a complex manifold $X$, and that $\rho$ is trivial (so that $f$ is a well defined univalued function on $X$). Then $f$ extends as a holomorphic function on the whole $X$.
\end{lemma}
\begin{proof}
    Straightforward consequence of Riemann extension Theorem.
\end{proof}

Let $X=B(0,r)$ be the open ball of radius $r$ centered at the origin of $\C^n$. Let  $H=\{x_1 \cdots x_s=0\}$  the union of some coordinate hyperplanes. Let $X_0=X - H$ and assume that $\rho:\pi_1(X_0,x_0)={\Z}^s\to \Aut(\D)$, the monodromy representation of $\F$, is non-trivial. The classical description of the centralizer of elements of $\Aut (\D)\simeq \Aut (\Hj)$ guarantees that the monodromy representation of $\F$ is of one of the following types.
\begin{enumerate}
    \item (Elliptic type) There exists $h\in \Aut(\D)$ such that for every $\gamma\in \pi_1(X)$, there exists $a_\gamma\in {\mathbb S}^1$ such that $h \circ \rho(\gamma) \circ h^{-1}(z)$ is equal to $R_{a_\gamma}(z) =  a_\gamma \cdot z$.
    \item (Parabolic type) There exists $h\in \Aut(\Pj^ 1)$, $h(\D)=\Hj$, such that for every $\gamma\in \pi_1(X)$, there exists $a_\gamma\in\R$ such that $h \circ \rho(\gamma) \circ h^{-1}$ is equal to $t_{a_\gamma}(z)= z+a_\gamma$.
    \item (Hyperbolic type) There exists $h\in \Aut(\Pj^ 1)$, $h(\D)=\Hj$, such that for every $\gamma\in \pi_1(X)$, there exists $a_\gamma\in(0,1) \cup (1,\infty)$ such that $h\circ \rho(\gamma) \circ h^{-1}$ is equal to $h_{a_\gamma}(z)=a_\gamma \cdot z$.
\end{enumerate}

\begin{thm}[Local extension]\label{TH:localextension}
    Notations as above. Let $f$ be a $\rho$-equivariant multivalued holomorphic mapping taking values in $\mathbb D$. If $f$ is non-constant then $\rho$ is not of hyperbolic type. Moreover, the following assertions hold true.
    \begin{enumerate}
        \item If $\rho$ is of parabolic type then there exists a holomorphic function $g$ on $X=B(0,r)$ and $s$ nonnegative real numbers $\lambda_1,....,\lambda_s$ such that
            \[
                h\circ f=\frac{1}{2i\pi}\sum_{i=1}^s \lambda_i \log{x_i} +g \, .
            \]
            in the Poincar\'e's upper half-plane model.
        \item If $ \rho$ is of elliptic type then there exists an holomorphic function $u : X \to \mathbb C$ and $s$ non-negative real numbers $\nu_1, \ldots, \nu_s\in [0,1) $such that
            \[
                h\circ f=u\prod {x_i}^{\nu_i} \, .
            \]
    \end{enumerate}
    In particular, the foliation $\F$ defined by the level sets of $f$ extends to a foliation on $X$ and also the transverse hyperbolic structure without poles on $X_0$ extends to a transverse hyperbolic structure (perhaps with poles) on $X$.
\end{thm}
\begin{proof}
    We first exclude the hyperbolic type. Let us argue by contradiction. It is enough to assume that $n=1$, i.e. $X = \mathbb D$ and $X_0 = \mathbb D^*$. This amounts to consider a multivalued function $f$ on $\D^*$ with monodromy generated by $h_{a}$, $a\not=1$. Consider the (possibly multivalued) function $\psi:\Hj\to \C$, $\psi (z)=z^{\frac{2i\pi}{a}}$ and remark that $\psi\circ h \circ f$ is a well defined (univalued) function on $\D^*$. Moreover, the image of $\psi$ lies in the annulus $\{u\in\C \, | \, e^{\frac{-2\pi}{a}}<|u|<1\}.$ Riemann extension theorem implies that $\psi\circ h \circ  f$ extends holomorphically on $\D$ as a unit $u$. One can write $u=e^{{2i\pi}{a} \varphi}$, where $\varphi$ is a holomorphic function on $\mathbb D$,  thus proving that  $f$ coincides (up to an additive constant) with $\varphi$. This contradicts the non-triviality of the monodromy.

    Assume now that $\rho$ is a representation of  parabolic type, assuming first that $n=1$. One can assume that $t_{_1} (z)=z+a$, $a\in \R -\{0\}$, is a generator of the monodromy. Set $\psi(z)=e^{\frac{2i\pi }{|a|}z}$ and remark that  $\psi\circ {h\circ f}$ is well defined as a  univalued function taking values in $\D^*=\psi (\Hj)$. The latter clearly implies that $h\circ f= \frac{1}{2i\pi |a|} \log{z} + g$,  $g$ holomorphic in $\D$ as wanted. To obtain the description of $h\circ f$ for $n>1$, it suffices to consider the restriction to one dimensional  linear subspaces obtained by fixing $n-1$ hyperplane coordinates and apply the previous argument. We leave the details to the reader.

    Concerning the elliptic type, the description follows easily from the following observation: there exists a holomorphic function $v$ on $X$ and non negative real numbers $\mu_i\geq 0$ such that $h\circ f=v\prod {x_i}^{\mu_i}$. Indeed, as $h\circ f$ takes values in $\D$, one gets in addition that $|v|\leq\frac{C}{\prod {|x_i|}^{\nu_i}}$ where $C$ is a suitable constant.  In particular, $v$ has a pole of order at most $[\nu_i]$ along $\{x_i=0\}$. Now, $u:= v\prod {x_i}^{[\mu_i]}$ and $\nu_i= \mu_i -[\mu_i]$ have the required properties.
\end{proof}

\begin{cor} \label{COR:regular singularities}
   Let $\F$ be a transversely hyperbolic foliation without poles on a quasi-projective manifold $X-H$ where the boundary divisor $H$ is a simple normal crossing divisor. Then $\F$ extends to a transversely projective foliation $\bar \F$ on $X$  with regular singularities. %Moreover, this projective  structure is necessarily unique whenever $\F$ is not algebraically integrable.
\end{cor}
\begin{proof}
    The explicit form of the developing map for $\F$ provided by Theorem \ref{TH:localextension} guarantees that the transverse projective structure for $\bar \F$ has moderate growth in the sense of  \cite[Chapter 2, Section 2]{MR0417174}, and therefore has regular singularities, cf. \cite[Chapter 2, Theorem 4.1 and Definition 4.2]{MR0417174}.
\end{proof}

\begin{cor}\label{COR:tangentcurve}
     Notations as above. Consider the extension $\bar \F$ of $\F$ on $X$ (which makes sense by virtue of Corollary \ref{COR:regular singularities}). %Assume that the union $K$ of the $\bar\F$-invariant algebraic hypersurfaces is a subset of $H$.
     Let $x\in X$.  Then the set of local separatrices ${ \mathcal S}_{i,x}$ through $x$  (i.e., the set of germs of irreducible  $\F$ invariant hypersurfaces passing through $x$) is finite.

      Moreover there exists a neighborhood of $x$, $V\subset X$  with the following properties:

      for every germ $\gamma:[0,\varepsilon)\to X$ of  curve tangent to $\bar\F$ and such that $\gamma(0)\in Y $, where $Y=\bigcup {\mathcal S}_{i,x}$,  %If its initial point $\gamma(0)$ belongs to  $K$  then
      the image of $\gamma$ is entirely contained in $Y$.
\end{cor}
\begin{proof}
    Using \ref{TH:localextension} %and supposing also that $H$ is $\F$ invariant,
    one always inherits locally a non constant multivalued  first integral of the form $F=U\prod {x_i}^{\nu_i}$ where $U$ is holomorphic and the $\nu_i$ are positive real numbers (this is of course obvious if $x$ does not belong to the polar locus of the structure). The function $|F|$ is thus genuinely locally defined   and its level sets $\{|F|=c\}$, $c\not=0$, do not accumulate on $Y:={ |F|}^{-1}(0)\ni x$. Thus $Y=\bigcup_i \{x_i=0\}\cup\{U=0\}$ is exactly the set of local separatrices at $x$.  The result follows.
\end{proof}

\subsection{Monodromy of  hyperbolic structures on quasi-projective curves}

\begin{prop}\label{monodromy curve in Shimura}
    Let $C_0$ be a smooth quasi-projective curve with smooth compactification $C$. Any (branched) hyperbolic structure on $C_0$ has Zariksi-dense monodromy.
\end{prop}
\begin{proof}
    Let $\rho : \pi_1(C_0, x_0) \to \Gamma \subset \Aut(\mathbb D)$ be the monodromy representation of the hyperbolic structure on $C_0$.
    After replacing $C_0$ by a finite \'etale covering we can assume that the Zariski closure $\overline{\Gamma }^\mathrm{Zar}$ is  connected. % \overline{\Gamma }^\mathrm{Zar}$. %Indeed, by Selberg's lemma there exists a finite index normal subgroup $\Gamma'\subset \Gamma$ which is torsion-free.
    Indeed, replacing $C_0$ be the covering determined by the kernel of the composition of $\rho$ with the quotient morphism $\Gamma \to \Gamma/\Gamma'$, $\Gamma'=\Gamma\cap { \overline{\Gamma }^\mathrm{Zar}_0}$, we obtain a quasi-projective curve with a hyperbolic structure which satisfies this property. Moreover, this new monodromy is Zariski dense if, and only if, the original one is Zariski dense.  

    By the structure of algebraic subgroups of $\Aut(\D) \subset \PSL(2,\C)$, there are three possible cases:
    \begin{enumerate}
        %\item $\Gamma$ is trivial;
        \item $\Gamma$ is conjugated to a subgroup of  $\Aff(\mathbb R) \subset \Aut(\mathbb H) \subset \PSL(2,\C)$;
        \item $\Gamma$ is conjugated to a subgroup of the group of rotations $S^1$;
        \item $\Gamma$ is Zariski dense.
    \end{enumerate}

%    Let $S = C- C_0$. We  first show that $\Gamma$ is infinite. If this were not the case, up to replacing $C$ with a finite cover (possibly ramified over $S$), we can assume that $\Gamma$ is trivial, so that the hyperbolic structure is defined by a holomorphic map
%    $C_0 \to \mb D.$ By the Riemann extension theorem, such map extends holomorphically through $S$, which leads to a contradiction.
%    Therefore, from now on we may assume that $\Gamma$ is infinite.

    Suppose now that $\Gamma$ is (conjugated to) a subgroup of $\Aff(\mathbb R)$. This means that one can pick local charts $F_i\colon U_i \to \mb H$ satisfying the gluing conditions $dF_i=a_{ij} dF_j$ with locally constant cocycles $a_{ij} \in \R$. In particular, the local holomorphic forms $dF_i$ glue to form a global section $\omega$ of $\Omega^1_{C}\otimes L$ on $C_0$ where $L$ is a numerically trivial line bundle. Theorem \ref{TH:localextension} implies that the local monodromy around points of $S$ is of parabolic type and, consequently,  this section extends on $C$ as a global section of $\Omega^1_{ C} (\log S) \otimes L$.

    Denote by $Z$ the zero divisor of $\omega$ induced by the branching points of the functions $F_i$. Then we obtain that the line bundle $\Omega^1_C(\log S) \otimes \mathcal{O}_C(-Z)$ is numerically trivial. But now, we observe that the Poincaré metric on $\mathbb H$ induces on $C_0$ a metric of negative curvature. Using Ahlfors-Schwarz lemma (see \cite{Cad16}), this metric extends on $C$ and induces a singular metric on the line bundle $\Omega^1_C(\log S) \otimes \mathcal{O}_C(-Z)$ whose curvature is a non-trivial semi-positive form. This shows that $\Omega^1_C(\log S) \otimes \mathcal{O}_C(-Z)$ cannot be numerically trivial. Therefore $\Gamma$ is not conjugated to a subgroup of $\Aff(\R)$.

    Finally, if  $\Gamma$ is (conjugated to) a rotation subgroup, one can pick local charts $F_i\colon U_i \to \mb D$ which are well-defined up to a multiplicative constant of type $e^{i\theta}$; in particular, the local logarithmic forms $dF_i/F_i$ glue to form a global logarithmic form $\omega$ on $C_0$. Remark that, if the polar locus is non-empty, then the residues of $\omega$ around poles are positive integers.

    Theorem \ref{TH:localextension} implies that $\omega$ extends to a logarithmic $1$-form globally defined on $C$. Moreover, the residues of such extension around the points of $S$ are all positive real numbers. The residue theorem implies that $\omega$ has no poles at all.

    Since we are assuming that $\Gamma$ is a rotation subgroup, the periods of $\omega$ are purely imaginary. Riemann's bilinear relations  implies that $\omega=0$, a contradiction. This concludes the proof of the proposition.
\end{proof}

\begin{cor} \label{COR:uniquenessstructure}
   Let $\F$ a transversely hyperbolic foliation on a projective manifold $X$. If $\F$ is not algebraically integrable then $\F$ admits a unique transversely projective structure.
\end{cor}
\begin{proof}
    Combine Proposition \ref{monodromy curve in Shimura} with Lemma \ref{L:uniqueness}.
\end{proof}
\begin{cor}\label{COR:trhyper}
	Let $\mathcal F$ be a transversely hyperbolic foliation on a projective manifold $X$. If $\mathcal F$ is not algebraically integrable
	then  $\mathcal F^{\trans}$, the transcendental part of $\mathcal F$, is also transversely hyperbolic.
\end{cor}
\begin{proof}
	Lemma \ref{L:CasaleCRAS} implies that $\mathcal F^{\trans}$ is transversely projective. The transverse projective structure for
	$\mathcal F^{\trans}$, cf. Definition \ref{def projective}, induces a transverse projective struture for $\mathcal F$. 
    Corollary \ref{COR:uniquenessstructure} implies that this structure must coincide with transversely hyperbolic structure for $\mathcal F$. It follows that the developing maps of the transverse hyperbolic structure of $\mathcal F$ are pull-backs of the developing maps of the quotient transversely projective structure for  $\F^{\trans}$. Hence there exists a developing map for $\F^{\trans}$ which takes values in Poincaré disk and 	we can conclude that $\F^{\trans}$ is transversely hyperbolic.
\end{proof}

\subsection{Transverse action for transversely hyperbolic foliations}

\begin{lemma}\label{L:hyperbolic and log-general}
    Let $X$ be a projective manifold and $\F$  be a  purely transcendental and transversely hyperbolic foliation on $X$ with polar divisor $\Delta$. Let $D$ be the reduced divisor with support equal to the union of the $\F$-invariant algebraic hypersurfaces and assume that $D$ is a simple normal crossing divisor. If $X_0 = X - D$ is of log general type then any dominant rational endomorphism $f:X \dashrightarrow X$ preserving $\F$ is a birational map, i.e.  $\Bir(\F) = \End(\F)$. Moreover, $\Bir(\F)$ is a finite group.
\end{lemma}
\begin{proof}
    Let $\xi\in {\mathcal P}_m:=H^0(X, (K_X(\log{D}))^{\otimes m} )$ be a logarithmic pluricanonical form. For $f\in \End(\F)$ one remarks that $f^*\xi\in  {\mathcal P}_m$. Indeed, if one assumes that this does not hold true then $f$ must contract a non invariant hypersurface $K$ such that $f(K)\subset D$. Pick a general point $m$ in $K$ where $\F$ and $K$ are in transverse position and denote by $V_m\subset K$ a small euclidean  neighborhood of $m$ in $K$. Consider the set of curves of the form $\gamma:[0,\varepsilon)\to X$ tangent to the foliation such that  $\gamma(0)\in V_m$. Because $f$ is generically a local diffeomorphism, this prevents from the existence of a local proper analytic subset $Y\ni f(m_0)$ containing the image of all $f\circ\gamma$.% one can choose $\gamma$ in such a way that $f\circ\gamma(t)$ does not meet  $D$ for $t\not=0$.
    This situation is excluded by Corollary \ref{COR:tangentcurve}.

    Via the linear and faithful action of $\End(\F)$ on ${\mathcal P}_m$ and the pluricanonical embedding, one inherits  a birational model of $(X,\F)$ where the action  of $\End(\F)$ is transferred into an action of a  linear algebraic group $G$ as explained in the proof of Theorem \ref{T:adj gen type} in Section \ref{S:adjoint}. This is sufficient to show that $f$ is a birational map.
    Since $f \in \End(\F)$ is arbitrary, we deduce that $\Bir(X,\F) = \End(X,\F)$.

    If the linear algebraic group $G$ is not finite as desired then, since it is Zariski closed, we can find a one-dimensional algebraic subgroup of it which preserves $(X,\F)$.  By assumption, it does not exist a non-trivial algebraically integrable subfoliation of $\F$. Thus $\F$ admits an infinitesimal (maybe rational) transverse symmetry, i.e., there exists a rational vector field $v$ such that $L_v\omega \wedge \omega = 0$ for any rational $1$-form defining $\F$.  Proposition \ref{P:transverse symmetries}(or rather its proof) implies that $\omega/\omega(v)$ is a closed $1$-form.   Thus $\F$ is transversely additive and has a projective structure different from  the transverse hyperbolic structure. This contradicts Corollary\ref{COR:uniquenessstructure}.
 \end{proof}

\begin{thm}\label{T:hyperbolic finite}
    Let $X$ be a projective manifold and let $\F$ be a transversely hyperbolic foliation on $X$ which is not algebraically integrable.
    Then the transverse action of $\End(X,\F)$ is finite.
\end{thm}
\begin{proof}
    Proposition \ref{P:reductiontotranscendental} and Corollary \ref{COR:trhyper} allow us to assume, without loss of generality,  that $\F$ is purely transcendental. One can also reduce to the case where the union of the $\F$-invariant hypersurfaces forms a \textit{snc} divisor $D$. Let $\mathfrak H$ be the quotient of the polydisc provided by Theorem  \ref{THM:D}. Denote by $\overline{\mathfrak{H}}$ the Baily-Borel compactification of ${\mathfrak{H}}$. This is a normal projective variety obtained by adding finitely many points (cusps).
    Theorem \ref{THM:D} implies the existence of a rational map $\pi: X \dashrightarrow \overline{ \mathfrak H}$ as well as a transversely projective foliation $\G$ on $\overline{ \mathfrak H}$ such that
    \begin{itemize}
    	\item The monodromy representation of $\G$ is described by one the tautological representation $\rho:\pi_1^{orb}(\mathfrak{H})\mapsto \PSL(2,\C)$
    	\item the monodromy representation of $\pi^*\G$ coincides with that of $\F$ on $X-D$.
    \end{itemize}
    %the monodromy representation of $\pi^*\G$ coincides with that of  %from $X$ through which factorizes as  (the projectivization of) one of the tautological representations $\rho:$    the monodromy  a polydisk Shimura modular
    %stack
     In particular, the restriction of  the monodromy representation of $\F$  to a  general fiber  of $\pi$ is trivial. Proposition \ref{monodromy curve in Shimura} implies that $\pi$ is generically finite. Let $\mathrm{Sing}(\overline{\mathfrak{H}})$ be the singular locus of $\overline{\mathfrak{H}}$. It consists of finitely many points, namely the elliptic points of $\mathfrak{H}$ plus the cusps. Let $p\in \mathrm{Sing}(\overline{\mathfrak{H}})$ and $W$ be a small Euclidean neighborhood of $p$. The image of the restriction of $\rho$ to $\pi_1(W-\{ p\})$ is at most affine. Invoking Proposition \ref{monodromy curve in Shimura} again, one can conclude that the fiber $\pi^{-1}(p)$ is $\F$-invariant. In particular, its codimension one part lies in $D$. Then, up to removing a codimension $\geq 2$ algebraic subset,  $X-D$ inherits from $\mathfrak H$ a polarized variation of Hodge structure with a generically injective period map. According to \cite[Theorem 1.1]{brunebarbe2017hyperbolicity}, this implies that $X-D$ is of log general type. The result follows from Lemma \ref{L:hyperbolic and log-general}.
\end{proof}

\subsection{Non-Kähler manifolds associated to number fields}
In this subsection, we will see that if the K\"ahler assumption is dropped, one can construct transversely hyperbolic foliations with non-finite transverse action. Here, by a Zariski dense set, we mean a set that is not contained  in any compact hypersurface.

Let us consider non-{K}ähler compact complex manifolds associated to number fields as constructed in \cite{OT} generalizing some examples of Inoue \cite{Ino}. Let $K$ be a number field, let $\sigma_1,\dots,\sigma_s$ be its real embeddings and $\sigma_{s+1},\dots,\sigma_{s+2t}$ its complex embeddings ($\sigma_{s+t+i}=\overline{\sigma_{s+i}}$). Assume $t>0$ (i.e. $K$ is not totally real). Let $\Hj$ be the Poincar\'e upper half-plane. Let $a \in \mathcal{O}_K$ act on $\Hj^s\times \C^t$ as a translation by the vector $(\sigma_1(a),\dots,\sigma_{s+t}(a))$. Let $u \in \mathcal{O}_K^{*,+}$ be a totally positive unit (i.e. $\sigma_i(u)>0$ for all real embeddings). Then $u$ acts on $\Hj^s\times \C^t$ by $u.(z_1,\dots,z_{s+t})=(\sigma_1(u)z_1,\dots, \sigma_{s+t}(u)z_{s+t}).$ For any subgroup $U$ of positive units, the semi-direct product $U \rtimes \mathcal{O}_K$ acts freely on $\Hj^s\times \C^t$. $U$ is called admissible if the quotient space $X(K,U)$ is a compact complex manifold. In particular, admissible groups must have rank $s$. One can always find such admissible subgroups.

From Dirichlet's units theorem, $\mathcal{O}_K^*$ is a group of rank $s+t-1$. Elements of $\mathcal{O}_K^{*,+}/U$ induce automorphisms of $X(K,U)$. Therefore as soon as $t>1$, one obtains automorphisms with infinite transverse order. Remark, that in the case of surfaces ($t=1$), elements of $\mathcal{O}_K^{*,+}/U$ are of finite order. Therefore such examples appear here only in dimension at least $3$. In particular, taking $s=1, t=2$, one obtains threefolds with transversely hyperbolic codimension $1$ foliations with infinite transverse action. Such manifolds are fibrations in $5$-dimensional (real) tori over the circle. The closures of the leaves of the foliation are $5$-dimensional tori and, therefore, Zariski dense. So such threefolds admit Zariski dense entire curves.

Here, the representation $\rho_\F:\pi_1 (X(K,U))\to \PSL(2,\R)$  associated to the transverse hyperbolic structure takes values in the affine subgroup $\Aff(2,\R)$ and its linear part $\rho_\F^1:\pi_1 (X(K,U))\to ({\R}_{>0},\times)$ has non-trivial image.

It is worth noticing that this situation cannot occur in the K\" ahler category. Indeed, suppose that $X$ is a compact K\"ahler manifold carrying a transversely hyperbolic  codimension one foliation which is also transversely affine and such that the linear part  $\rho_\F^1:\pi_1 (X)\to ({\R}_{>0},\times)$ has non-trivial image. To wit, there exists on $X$ an open cover $(U_i)$ such that for every $i$, $\F$ is defined by $dw_i=0$, where $w_i:U_i\to\Hj$ is submersive on $U_i -\mbox{Sing}\ \F$ and such that the glueing conditions $w_i=\varphi_{ij}\circ w_j$ are defined by locally constant elements $\varphi_{ij}$ of $\Aff(2,\R)$. In particular there exists locally constants cocycles $a_{ij}\in{\R}_{>0}$ such that
\[
    dw_i=a_{ij}dw_j
\]
and the normal bundle $N_\F$ is thus numerically trivial. On the other hand, the existence of a transverse hyperbolic structure directly implies that  $N_\F$ is equipped with a metric whose curvature  is a non-trivial  semi-negative form. This shows that $c_1(N_\F)\not=0$, a contradiction.

% The non-triviality of $\rho_\F^1$ can be translated into the fact that the cocycle $\{a_{ij}\}$ is not zero in $H^1(X, {\C}^*)$. Let $\nabla_1$ the flat connection on $N_\F$ whose local horizontal sections are given by the local system $\C_{\rho_\F^1}$ and $\omega$ be the (unique up to multiplicative constant) section of $\Omega_X^1 \otimes N_\F$ defining $\F$. From the glueing conditions \ref{E:locconstantglueing}, one can infer that $\nabla_1 (\omega)=0$. On the other hand, standard Hodge identities yield $\nabla_2 (\omega)=0$, where $\nabla_2$ is the unique unitary connection on $N_\F$. These closedness properties thus give rise to a non-trivial holomorphic one form, namely $\eta=\nabla_1 -\nabla_2$ such that $\eta\wedge\omega=0$. This implies that $\omega=f\eta$ where $f$ is a non-trivial section of $N_\F$. In particular, $N_\F$ is necessarily trivial  and $f$ is nowhere vanishing. In the trivializing charts $U_i$, this section is represented by $f_i\in{\mathcal O}^* (U_i)$, such that $f_i=a_{ij} f_j$. The representation $\rho_\F^1$ is thus given by the exponential of the periods of the holomorphic one form $\xi=\nabla_1 (f) /f$ ($=df_i/f_i$). In particular, $\xi$ is non-trivial and has purely imaginary periods, which is impossible.
%

\subsection{Monodromy determines  hyperbolic structures}

\begin{thm}\label{T:Uniqueness}
    Let $C_0$ be a smooth quasi-projective curve contained in a smooth projective curve $C$. Let $\rho:\pi_1(C_0)\rightarrow \Aut (\Hj)\simeq\Aut (\D)$ be a morphism. Then there exists at most one $\rho$-equivariant and non constant holomorphic function $f:\tilde C_0\to \Hj\simeq \D$ where $\tilde C_0$ is the universal cover of $C_0$.
\end{thm}

When $C=C_0$, Theorem \ref{T:Uniqueness} is not completely original, as it can be deduced from general properties of equivariant harmonic maps, see for instance \cite{MR965220} and references therein. We provide here a simple proof that is not extracted from the current literature on these topics. When $C\neq C_0$, the proof is considerably more involved.

\begin{proof}[Proof of Theorem \ref{T:Uniqueness}]
%    We will first consider the projective case, i.e. $C_0 = C$ is a smooth projective curve.
    There is a classical construction, the suspension process, which associates to a group morphism $\pi_1(C_0)\to PSL_2(\C)$ a $\Pj^1$-bundle  over $C_0$ and a foliation $\F$ transverse to the fibers of the $\mathbb P^1$-bundle. It goes as follows: $\pi_1(C_0)$ acts on the product $\tilde C_0\times \Pj^1$ by $(\gamma \cdot \tilde x, \rho(\gamma) \cdot z)$. This action is free and properly discontinuous, hence the quotient $M_\rho=\tilde C_0\times \Pj^1 /\pi_1(C_0)$ is a manifold via the projection $(\tilde x,z) \mapsto \tilde x$ and whose fibers are copies of $\Pj^1$.

    In our setting, $\rho$ takes values in $PSL_2(\R)=\Aut (\Hj)$, then $M_\rho$ canonically contains a flat $\Hj$-bundle over $C_0$ that we will denote by $H_\rho$.  Let $\mathcal V$ be the vertical foliation on $H_{\rho}$ defined by the fibers of this $\Hj$-bundle. The $PSL_2(\R)$ invariance of the Poincar\'e metric on the disc guarantees the existence of a closed semi-positive $(1,1)$-form $\xi$ on $H_{\rho}$ such that its restriction to any leaf of $\mathcal V$ coincides with the Poincaré metric  and the kernel of $\xi$ defines the {\it horizontal} foliation $\F$. The curvature form of this metric is equal (up to some positive factor) to $-\xi$ and
    represents $c_1(T_{\mathcal V}) = c_1(N_{\F})$ on $H_{\rho}$.

    Now, the existence of a $\rho$-equivariant holomorphic function $f$ as above is equivalent to the existence of a section $\sigma:C_0\to M_\rho$ whose image lies in $H_\rho$: namely, $\sigma([\tilde x])=[\tilde x, f(\tilde x)]$. The restriction of $T_{\mathcal V}$ to (the image of) the section $\sigma$ coincides with $N_\sigma$, the normal bundle of $\sigma$. The metric on $T_{\mathcal V}$  induces, by restriction, a metric $h$ on $N_\sigma$. % the normal bundle of $\sigma$.
    Since $\sigma$ is not invariant by $\F$, the restriction of $\xi$ to $\sigma$ is a nonzero semi-positive $(1,1)$-form.

     We will first consider the \textit{projective case}, i.e. $C_0 = C$ is a smooth projective curve.
     It follows that $\sigma^2 <0$, since
    \[
        \sigma^2 = \int_{\sigma} c_1(N_{\sigma}) = \int_{\sigma} - \xi < 0 \, .
    \]
   According to \cite[Proposition 11, p.122]{MR1600388}, the Neron-Severi group of $M_\rho$ is generated by a class of any section, for instance $[\sigma]$ and the class $[F]$ of a vertical fiber so that there is no section $s$ distinct from $\sigma$ %The Hodge index Theorem that there is no section
   whith self-intersection $<-\sigma^2$. Indeed, if $s$ is a section, one has write $[s]=[\sigma]+\alpha[F]$ where $\alpha$ is a real number $\geq -\sigma^2$ whenever $s\not=\sigma$. %lies in $(\sigma^2, -\sigma^2)$.
   This implies  %the uniqueness of $\sigma$ and consequently
   the uniqueness of $f$ in the case $C_0= C$.

    In general, $C_0 \neq C$ and the $\Pj^1$-bundle $M_\rho$ as well as the section $\sigma$ fail to be compact. However, it is possible to compactify the total space of $M_{\rho}$, together with the section and the foliation. We will denote by $\bar{M_\rho}$, $\bar{\sigma}$ and $\bar{\F}$ the respective compactifications. It turns out that $\bar{M_\rho}$ has still the structure of a $\Pj^1$-bundle transverse to $\bar{\F}$ except over compactification points. Let us detail the construction of such a compactification.

    Let $p\in \{p_1,..., p_n\}=C-C_0$ one of the punctures.  Here, we will regard $f$ as a multivalued map $C_0\to \Hj\subset \Pj^1$ with non-trivial local monodromy around $p$ by virtue of Lemma \ref{L:Rextension}. According to  Theorem \ref{TH:localextension} there exists a  local coordinate $u$, $u(p)=0$, such that
    \begin{enumerate}
        \item (parabolic)  $f= g(\frac{1}{2i\pi}\log{u})$; or
        \item (elliptic) $f=h(u^\nu)$, $\nu>0$ ,
    \end{enumerate}
    where $g\in \Aut (\Hj)$ and $h$ is an homography such that $h(\mathbb D)=\Hj$.

    Let us first deal with the case of a parabolic puncture. Consider the quotient $\Hj\times \Pj^1$ by the transformation $\varphi(\tau, z)=(\tau+1, z+1)$. One can identify the quotient  space with $\D^*\times \Pj^1$ equipped with standard coordinates $(u,v)$.  Actually, the $(\tau,z)$ coordinates are related to the $(u,v)$ coordinates via
    \[
        u=e^{2i\pi\tau}, \quad v=z-\tau.
    \]
    Note that $\Hj\times\Hj\subset \Hj\times \Pj^1$ is $\varphi$ invariant and that the quotient $\Hj\times\Hj/\varphi$ identifies to the open set $U$ of $\D^*\times \Pj^1$ defined by the inequality  $\im{(v+\frac{1}{2i\pi}\log{u})}>0$. As the transformation $\varphi$ acts diagonally,  the horizontal foliation $\mathcal{H}=Ker(dz)$ on $\Hj\times \Pj^1$ descends on $\D^*\times \Pj^1$ as the regular  foliation $\F$ defined by the $1$-form
    \begin{equation}
        {dv} + \frac{1}{2i\pi}\frac{du}{u}\, , \quad  (u,v)\in \D\times \Pj^1 \, .
    \end{equation}
    Hence $\F$  extends as a singular foliation $\bar{ \F}_p$ on the natural compactification $\D\times\Pj^1$.

    This foliation is transverse to every fiber $\{u\}\times \Pj^1$ at the exception of the central fiber $\{0\}\times\Pj^1$ which is invariant. Note that $\bar{\F}_p$ admits a unique singularity $q$ at $(0,\infty)$, namely a saddle-node defined by $\frac{dV}{du}=\frac{V^2}{2i\pi u}$, $V=\frac{1}{v}$. The only leaf of $\mathcal{H}$ fixed by $\varphi$, namely $\{z=\infty\}$, corresponds to the strong separatrix $\{V=0\}$ of $\bar{ \F}_p$ at the singularity $q$.

    Because $\varphi$ acts by isometry on the second factor, the (vertical) Poincar\'e metric defined on $\Hj\times\Hj$ by the form
    \[
        {i}\dfrac{dz\wedge d\bar{z}}{{(\im{z})}^2}
    \]
    descends to $U$ and induces on the tangent bundle to the vertical foliation defined by $\partial_v$ on $U$ a metric
    \[
        h(u, \partial_v)=\dfrac{1}{{(\im{(v+\frac{1}{2i\pi}\log{u})})}^2}.
    \]

    Now consider a holomorphic function $\sigma:\Hj\mapsto\Hj$ such that $\sigma(\tau+1)=\sigma(\tau) +1$. One can regard $\sigma$ as a section of the previous trivial  $\Pj^1$ bundle over $\D^*$ such that the image of $\sigma$ lies in $U$. Note then that $\sigma$ reads as
    \[
        \sigma(u)=\varphi (u)- \frac{1}{2i\pi}\log{u}
    \]
    where $\varphi(u)$ is a multivalued function taking values in $\Hj$ (hence whose lift on $\Hj$ satisfies the same equivariance property as $\sigma$). Theorem \ref{TH:localextension} implies that $\sigma$ extends through $0\in \D$ as a holomorphic function $\bar{\sigma}: \D\mapsto \C$ (in particular the foliation $\F$ is regular at $\sigma(0)$. The normal bundle of $\bar{\sigma}$ is equipped with the metric $h$ (more exactly its restriction) singular at $u=0$ and which reads as $h(u,\partial v)= e^{-2\varphi}$ where the weight $\varphi$ is defined as
    \[
        \varphi(u)= \log{(-{\frac{1}{2\pi}\log |u e^{2i\pi\sigma(u)}|})}\, .
    \]
    A  calculation shows that $-\varphi$ is subharmonic. Indeed, the curvature current $T_h=\frac{i}{2\pi}\partial\bar{\partial}\varphi$ of $h|N_{\bar{\sigma}}$ is negative. Actually, up to some positive constant,  $T_h$ %$h|N_{\bar{\sigma}}$
    is equal to
    \[
        - i\dfrac{1}{{|U|}^2{\log^2{|U|}}}dU\wedge d\bar{U}
    \]
    where $U=u e^{2i\pi\sigma(u)}$.

    We now investigate the existence of similar equivariant compactifications for singularities of elliptic type. We first start from the following model (where it is more convenient to consider the disk instead of the Poincar\'e upper half plane): quotient of the horizontal foliation on $\Hj\times \Pj^1$ by the tranformation $\varphi(\tau, z) = (\tau+1, e^{-2i\pi \theta}z)$, $\theta>0$. As before, this quotient is isomorphic to the trivial $\Pj^1$ bundle over $\D^*$ equipped with standard coordinates $(u,v)$ related to the $(\tau,z)$ coordinates via $u=e^{2i\pi\tau}$, $v=ze^{2i\pi\theta\tau}$. Note that $\Hj\times\D\subset \Hj\times \Pj^1$ is $\varphi$ invariant and that the quotient $\Hj\times\D/\varphi$ identifies to the open set $U$ of $\D^*\times \Pj^1$ defined by the inequality $|vu^{-\theta}|<1$. By construction the horizontal foliation $\mathcal H=\ker(dz)$ descends to $\D^*\times\Pj^1$ as the regular foliation $\F$ defined by the $1$-form
    \[
        \frac{dv}{v} -\theta \frac{d\tau}{\tau}.
    \]
    This foliation extends as a singular foliation $\bar{\F}_p$ on the natural compactification $\D\times \Pj^1$. Note that $\bar{\F}_p$ is tangent to the central fiber $(0,v)$ and admits on it two linearizable singularities at $v=0$ and $v=\infty$.  The  leaves of $\mathcal{H}$ fixed by $\varphi$, namely $\{z=0,\infty\}$, corresponds to the  separatrices $\{v=0,\infty\}$ of $\bar{ \F}_p$ along the central fiber $\{u=0\}$.

    Because $\varphi$ acts by isometry on the second factor, the (vertical) Poincar\'e metric defined on $\Hj\times\D$ by the form
    \[
        {i}\dfrac{dz\wedge d\bar{z}}{{(1-{|z|}^2)}^2}
    \]
    descends to $U$ and induces on the tangent bundle to the vertical foliation defined by $\partial_v$ on $U$ the metric
    \[
        h(u, \partial_v)=\dfrac{{|u|}^{-2\theta}}{{(1-{|vu^{-\theta}|}^2))}^2}.
    \]

    Now consider a holomorphic function $\sigma:\Hj\mapsto\D$ such that $\sigma(\tau+1)=e^{-2i\pi\theta}\sigma(\tau) $. One can regard $\sigma$ as a section of the previous trivial  $\Pj^1$ bundle over $\D^*$ such that the image of $\sigma$ lies in $U$. Note then that $\sigma$ reads as
    \[
        \sigma(u)=\varphi (u) u^{\theta}
    \]
    where $\varphi(u)$ is a multivalued function taking values in $\D$. % (hence whose lift on $\Hj$ satisfies the same equivariance property than $f$).
    From \ref{TH:localextension}, one deduces that $\sigma$ extends through $0\in \D$ as a holomorphic function $\bar{\sigma}: \D\mapsto \C$. Indeed, recall that $\varphi$ takes the form $\varphi(u)= U(u) u^\nu$ where $U:\D\mapsto \C$ is holomorphic, $\nu>0$ and $\nu\equiv -\theta [2\pi]$. Unlike the parabolic case, $(0,\bar{\sigma}(0))$  coincides  with the singular point $(0,0)$ of $\bar{ \F}_p$. The normal bundle of $\bar{\sigma}$ is equipped with the metric $h$ (more exactly its restriction) singular at $u=0$ and which reads as $h(u,\partial v)= e^{-2\varphi}$ where the weight $\varphi$ is defined as $-\theta\log{|u|}-\log(1-{|\varphi (u)|}^2)$.

    Again, a computation shows that $-\varphi$ is subharmonic . More precisely, the curvature current $T_h=\frac{i}{2\pi}\partial\bar{\partial}\varphi$ of $h|N_{\bar{\sigma}}$ reads as
    \begin{equation}\label{E:ellipticbehavior}
        T_h= -\theta\delta_{ \{u=0\}} +\eta
    \end{equation}
    where $\eta$ is a semi-negative and non-trivial $(1,1)$-form with $L_{loc}^1$ coefficients. Unlike the parabolic case, the curvature current admits an atomic part which has a negative contribution to the curvature of $N_{\bar{ \sigma}}$. This explains \textit{a posteriori} the choice of this compactification by specifying the choice of the \textit{ elliptic weight} $\theta>0$.

    For every puncture $p$, consider a small neighborhood of $p$  biholomorphic to $\D\subset C$. Let $\pi:M_\rho\mapsto C_0$ the canonical projection. One can glue over $\pi^{-1}(\D^*)$ one of the two previous local models, depending on whether $p$ is elliptic or parabolic. Doing this for every $p$, one ends-up with the desired compactification $\overline{M_\rho}$.

    By the foregoing construction, any $\rho$-equivariant (and non constant) map $f$, corresponds to a section $\bar{\sigma}: C\to \overline{M_\rho}$ with $\sigma^2<0$. Indeed, the metric $h$ on $N_{\bar{ \sigma}}$ arising from the vertical Poincar\'e metric is invariant by this gluing process. On the other hand, invoking again \cite[Proposition 11, p.122]{MR1600388},  the Neron-Severi group is generated by the class of $\bar{ \sigma}$ and that of a vertical fiber, so that every section different from $\bar{ \sigma}$ has self-intersection $\geq -{ \bar\sigma}^2$. As in the projective case, this proves the uniqueness of $f$.
\end{proof}
\subsection{Additional properties}

	Let $f:\tilde C_0\to \Hj$ as in the statement of Theorem \ref{T:Uniqueness}. One also maintains notations of its proof.
    Let $\theta_i>0$, $i=1....,e$, some \textit{ elliptic weights} associated to the elliptic punctures. From the above analysis (in particular ( \ref{E:ellipticbehavior}), one can infer that ${ \bar\sigma}^2< -\sum \theta_i$.

\begin{rem}\label{R:altproof}[Alternative proof of Proposition \ref{monodromy curve in Shimura}]
    Suppose now that $\rho$ is not Zariski dense. Then, up to taking a finite \'etale cover of $C_0$, there is no loss of generality in assuming that the image of $\rho$  fixes a point in $\Pj^1$. From the description of the singularities of the compactified foliation $\bar{\F}$ of $\F$ (the "horizontal" foliation defining the flat structure), one can derive the existence of a section $s:\bar C\to \overline{M_\rho}$,  $\bar{ \F}$ invariant with  self-intersection $s^2\leq \sum \theta_i$. Actually the Camacho-Sad indices (whose sum gives $s^2$, see \cite[Theorem 3.2, p.28]{ MR3328860}) vanishes over parabolic puncture and are equal to $\pm\theta_i$ over elliptic punctures). This contradicts the estimate on the self-intersection of sections different from $\bar{ \sigma}$.
\end{rem}

\begin{rem}\label{R:nonsimult}[Non existence of a simultaneous "conjugated" holomorphic section]
    Suppose that there also exists a $\rho$-equivariant holomorphic map  $g :\tilde C\to\Pj^1- \overline{ \Hj}$. By the proof of Theorem  \ref{T:Uniqueness} and considering this times the Poincar\'e metric on the flat $\Pj^1- \overline{ \Hj}$ bundle over $C_0$, this map gives rise to a section $\bar{\alpha}:C\mapsto \overline{ M_\rho}$ (different from $\bar{ \sigma}$) of self-intersection $\leq\sum \theta_i$ ($-\theta$ is replaced by $\theta$ in Equation \ref{E:ellipticbehavior}). As in Remark \ref{R:altproof}, we obtain a contradiction.
\end{rem}

\subsection{Representations and rational endomorphisms}

\begin{thm}\label{TH:uniquemonodromy}
    Let  $X$ be a projective manifold and $\F$ be a (singular) transversely hyperbolic foliation on $X$ with polar divisor $\Delta$. Let $\rho:\pi_1(X- \Delta)\rightarrow \Aut (\D)$ be the corresponding monodromy representation. Assume that there exists a rational endomorphism $f: X \dashrightarrow X$ such that $f^*\rho= \rho$ modulo conjugation in $\Aut(\Pj^1)$ (that is $\rho$ is the monodromy of the pull-back foliation $f^*\rho$) Then, $f^*\F=\F$.
\end{thm}
\begin{proof}
    Pick a general curve $C \subset X$. Note that both $\F$ and $\G=f^*\F$ are transversely hyperbolic foliations on $X$ and their developping maps are locally defined on $X-\Delta$ (the latter by Lemma \ref{L:Rextension}).
    Furthermore, the monodromy of $\G$ is conjugated to the monodromy of  $\F$ in $\Aut (\Pj^1)$. By Zariski density (Proposition \ref{monodromy curve in Shimura}), the conjugacy lies in $\Aut(\D)$ or $\tau\circ \Aut (\D)\circ\tau^{-1}$ where $\tau(z)=\frac{1}{z}$. %In both cases, the foliation $\tilde \F= {(f\circ f)}^*\F$ has the same monodromy of $\F$ up to conjugation in $\Aut(\D)$.

    Each of these foliations induce on $C-\Delta$ a (branched) hyperbolic structure defined by  $\rho_{|\pi_1(C-\Delta)}$-equivariant holomorphic maps
    \begin{itemize}
    	\item $F_\F, F_{\G}:\widetilde{C-\Delta}\to \D$
    	
    	or
    	
    	\item $F_\F :\widetilde{C-\Delta}\to \D$, $F_{\G}:\widetilde{C-\Delta}\to \tau(\D)=\Pj^1-\bar\D$
    \end{itemize}
     %that can be alternatively regarded as multivalued maps on $C-\Delta$.
     The conclusion follows from Theorem \ref{T:Uniqueness} and Remark \ref{R:nonsimult} by varying $C$ \textit{ ad libitum}.
\end{proof}

%\begin{rem}
%    One could go into further details in the previous proof by showing that one has directly $\G:=f^*\F=\F$ in the statement of Theorem \ref{TH:uniquemonodromy}. Indeed, assume for simplicity that $C$ does not meet $\Delta$. From the construction  of the flat $\Pj^1$ bundle (with structural group in $\Aut(\D)$) over $C$ performed in the proof of \ref{TH:uniquemonodromy}, one get \textit{two  sections} $s_\F,s_\G:C\to M_\rho$ which respectively take values in the  associated $\D$ bundle and $\Pj^1-\overline\D$ bundle both with \textit{negative self-intersection} (use the Poincar\'e metric in $\Pj^1-\overline\D$ for $s_2$).  This contradicts the uniqueness of such sections.
%\end{rem}

\section{Proof of Theorem \ref{THM:transv proj}}\label{S:THM A}

\subsection{Polydisk Shimura modular orbifolds and their tautological foliations}
Recall that an \emph{orbifold} is a Hausdorff topological space which is locally modeled on finite quotients of $\C^n$. One defines an orbifold cover as a map $f\colon X\to Y$ between orbifolds which is locally conjugated to a quotient map
\[
    \C^n/\Gamma_0 \to \C^n/\Gamma_1 \qquad \Gamma_0 \leq \Gamma_1.
\]
 Given an orbifold $X$, a result due to Thurston asserts there exists a universal orbifold cover $\pi\colon \widetilde X \to X$; the \emph{orbifold fundamental group} $\pi_1^{orb}(X)$ is then defined as the group of deck transformations of $\pi$. For example, if $U$ is a simply connected complex manifold and $G\leq \Aut(U)$ is a discrete subgroup such that the stabilizer of each point of $U$ is finite, then the quotient $X=U/G$ admits a natural orbifold structure such that $\pi_1^{orb}(X)=G$.

    Following Corlette and Simpson \cite{MR2457528}, a polydisk Shimura modular orbifold is a  quotient $\mathfrak H$ of  a polydisk $\mathbb D^n$ by a  group of the form $U(P,\Phi)$ where $P$ is a projective module of  rank two over the ring of integers $\mathcal O_L$ of a  totally imaginary  (quadratic extension) $L$ of a totally real number field $F$; $\Phi$ is a skew hermitian form on $P_L=P \otimes_{\mathcal O_L} L$; and $U(P,\Phi)$  is the subgroup of the $\Phi$-unitary group  $U(P_L,\Phi)$
    consisting of elements that preserve $P$. This group acts naturally  on $\mathbb D^n$ where $n$ is half the number of
    embeddings $\sigma : L \to \C$ such that the quadratic form $\sqrt{-1} \Phi(v,v)$ is indefinite. The aforementioned action is explained in detail  in \cite[\S 9]{MR2457528}.
    The quotients $\mathbb D^n /\mbox{U}(P,\Phi)$ are always  quasiprojective
    orbifolds, and when $[L:\mathbb Q] > 2n$, they are projective (i.e., proper/compact) orbifolds.
    The archetypical examples satisfying $[L:\mathbb Q] =2n$ are the Hilbert modular orbifolds, which are quasiprojective but
    not projective.

    Note that there is one tautological representation
    \[
         \pi_1^{orb}(\mathbb D^n /\mbox{U}(P,\Phi) ) \simeq \mbox{SU}(P,\Phi)/\{\pm \Id\} \hookrightarrow \PSL(2,L) \, ,
    \]
    which induces for each embedding $\sigma: L \to \C$ one tautological representation  $\pi_1^{orb}(\mathbb D^n /\mbox{U}(P,\Phi) ) \to \PSL(2,\C)$. If $\rho:\pi_1^{orb}(\mathbb D^n /\mbox{U}(P,\Phi) ) \to \PSL(2,\C)$  is one such tautological representation then
    we will denote by $\mathcal H_{\rho}$ the associated Riccati foliation on the $\mathbb P^1$-bundle $\mathfrak H\times_\rho\mathbb P^1$.
    Varying the embedding $\sigma : L \to \C$, we obtain the set of tautological Riccati foliations over $\mathfrak H$.

\subsection{Structure}\label{SS:structure} The structure of transversely projective
foliations on projective manifolds was described in \cite{MR3522824}. We recall this description below.

\begin{thm}\label{THM:D}
    Let $\mathcal F$ be a codimension one transversely projective
    foliation on a projective manifold $X$.
    Then at least one of the following assertions
    holds true.
    \begin{enumerate}
        \item\label{I:D1} The foliation $\F$ is virtually transversely additive.
        \item\label{I:D2} There exists a rational dominant map $f: X \dashrightarrow S$ to a ruled surface $\pi:S\to C$, and a Riccati foliation
        $\mathcal H$ defined on $S$ (i.e. over the curve $C$) such that $\mathcal F= f^* \mathcal H$.
        \item\label{I:D3} There exists a polydisk Shimura modular orbifold $\mathfrak H$ and a rational map $\varphi: X \dashrightarrow \mathfrak H\times_\rho\mathbb P^1$ such that  $\mathcal F= f^* \mathcal H_\rho$ where $\mathcal H_\rho$ is one the tautological Riccati foliations over $\mathfrak H$.
    \end{enumerate}
\end{thm}

Note that one can assume that   foliations fitting in the description  by provided by  (\ref{I:D2}) and (\ref{I:D3}) are not virtually additive (cse covered by Item (\ref{I:D1}))  and in particular %Items (\ref{I:D1}), so that one can assume that the foliations described in Items (\ref{I:D2}) and (\ref{I:D2})
have transcendental dimension at most  one.

\subsection{Proof of Theorem \ref{THM:transv proj}}
Let $\F$ be a transversely projective foliation on a projective manifold $X$. By Proposition \ref{P:reductiontotranscendental} and Lemma \ref{L:CasaleCRAS}, we can assume without loss of generality that $\F$ is purely transcendental.
Theorem \ref{THM:D} says that $\F$ is (\ref{I:D1}) virtually transversely additive; or (\ref{I:D2}) $\F$ is the pull-back of Riccati foliation on a projective surface; or (\ref{I:D3}) there exists a polydisk Shimura modular orbifold $\mathfrak H$ and a rational map $\varphi: X \dashrightarrow \mathfrak H\times_\rho\mathbb P^1$ such that  $\mathcal F= \varphi^* \mathcal H_\rho$ where $\mathcal H_\rho$ is one the tautological Riccati foliations of $\mathfrak H$.

If $\F$ satisfies (\ref{I:D1}) then there is nothing to prove. If instead $\F$ satisfies (\ref{I:D2}) then the conclusion follows from Corollary \ref{C:reductiontranscendental}.

It remains to treat the case (\ref{I:D3}). For that, consider the diagram
\begin{center}
    \begin{tikzcd}[column sep=small]
    X  \arrow[d,"p"] \arrow[rr, dashed, "\varphi" ] \arrow[drr,dashed,"\pi \circ \varphi"] & & \mathfrak H\times_\rho\mathbb P^1 \arrow[d,"\pi"] \\
    B \arrow[rr, "q"] & & \bar{\mathfrak H}
    \end{tikzcd}
\end{center}
where $ \bar{\mathfrak H} $ is projective compactification of  $\bar{\mathfrak H}$ and  $q: B \to \bar{ \mathfrak H}$  is  a Stein factorization of a desingularization of the closure of the image of $\pi\circ \varphi$. After replacing $X$ and $B$  by a suitable birational model, we can assume that the natural map from $X$ to $B$ is a morphism  $p : X \rightarrow B$ between projective manifolds as depicted above. Note that the general fiber of $p$ is irreducible, since $q$ is a Stein factorization.

The identity $\F = \varphi^* \mathcal H_{\rho}$ implies that the general fiber of $\varphi$ is finite since we are assuming that $\F$ is purely transcendental.  Thus, the general fiber of $p$ is at most one dimensional.

Let $\Delta$ be the polar divisor of the transverse structure for $\F$ and let $\Delta_B = p_* \Delta$  be its image on $B$. The monodromy representation
of $\F$, $\rho_{\F} : \pi_1(X - \Delta) \to \Aut(\mathbb P^1)$ factors through a representation of $\rho_B: \pi_1(B - \Delta_B,) \to \Aut(\mathbb P^1)$. Corlette and Simpson's description of the rank two representations implies that a Galois conjugate of $\rho_B$ is conjugated to the monodromy
of a transversely hyperbolic foliation $\G$ on $B$ induced by one of the $\dim \mathfrak H$  natural transversely hyperbolic foliations on $\mathfrak H$.

Any rational endomorphism $f: X \dashrightarrow X$ of $\F$ preserves the monodromy representation of $\F$ since, under our assumptions, Lemma \ref{uniqueness projective structure} guarantees that $\F$ carries a unique transverse projective structure. Therefore Proposition \ref{monodromy curve in Shimura} implies that $f$ must preserve the set of fibers of $p$. Hence, since $p$ has irreducible general fiber, there exists  a rational map $g: B \dashrightarrow B$
such that the diagram
\begin{center}
    \begin{tikzcd}[column sep=small]
    X  \arrow[d,"p"] \arrow[rr, dashed, "f" ] & & X \arrow[d,"p"]\\
    B \arrow[rr, dashed, "g"] & & B
    \end{tikzcd}
\end{center}
commutes. Moreover, $g$ preserves the monodromy represention $\rho_B$ induced by $\rho$. Theorem \ref{TH:uniquemonodromy} implies that $g$ is a rational endomorphism of the transversely hyperbolic foliation $\G$. We can argue as in the proof of Theorem \ref{T:hyperbolic finite}, using \cite[Theorem 1.1]{brunebarbe2017hyperbolicity} and Lemma \ref{L:hyperbolic and log-general}, to deduce that $\Bir(B,\G) = \End(B,\G)$ is a finite group. Consequently, the Zariski closure of the orbits of $\End(\F)$ is either finite or a finite union of fibers of $p$. In the first case $\End(\F)$ is finite and we are done. In the second case, Proposition \ref{P:dim orbit 1 bis} implies that $\F$ is virtually transversely additive contrary to our assumptions.
\qed

\section{Quasi-Albanese morphism and Zariski dense dynamics}\label{S:qalbanese}

\subsection{Endomorphisms of semi-abelian varieties}

The following property of self-maps of semi-abelian variety is well known, see for instance  \cite[Fact 2.1]{MR3937327}  and references therein.

\begin{lemma}\label{L:endomorphisms}
    Let $\varphi: G \to G$ be a  morphism of a semi-abelian variety.
    Then $\varphi$ can be written as the composition of a group endomorphism of $G$
    with a translation. In particular, if the set of fixed points of $\varphi$ is non-empty
    then $\varphi$ is conjugated to a group endomorphism of $G$ and, after conjugation, its set of fixed points
    is a semi-abelian subvariety.
\end{lemma}

\begin{prop}\label{P:zariski closure}
    Let $G$ be a semi-abelian variety and let $\varphi : G \to G$ be a dominant morphism.
    Then, for any point $x\in G$, the Zariski closure of the forward orbit of $x$ is a
    finite union of translated semi-abelian subvarieties of $G$.
\end{prop}
\begin{proof}
    This result  appears in \cite[Fact 2.9]{MR3937327}. For the reader's convenience, we sketch here its proof.
    The key observation is that the orbit of a point $x \in G$ is contained in a finitely generated
    subgroup $\Gamma$. If $V$ is the  Zariski closure of the orbit then the intersection $V \cap \Gamma$ is
    Zariski dense in $V$. Vojta's confirmation of Mordell-Lang conjecture (see \cite[Fact 2.8]{MR3937327} or
     \cite{MR1408559}) guarantees that $V$ is of the alleged
    form.
\end{proof}

\begin{cor}\label{C:invariant hypersurface}
    Let $G$ be a semi-abelian variety and let $\varphi : G \to G$ be a morphism with a Zariski dense orbit.
    If $V \subset A$ is an irreducible hypersurface invariant by $\varphi$ then $V$ is a translate of a semi-abelian subvariety.
\end{cor}
\begin{proof}
    Consider the restriction of $\varphi$ to $V$. If $\varphi_{|V} $ has finite order then $V$ is contained in the set of fixed points of some iterate of $\varphi$. Lemma \ref{L:endomorphisms} implies that $V$ is the translate of a semi-abelian subvariety as claimed.

  From now on, assume that $\varphi_{|V}$ has infinite order and let $x \in V$ be a very general point of $V$. Proposition \ref{P:zariski closure} implies that the Zariski closure of the orbit of $x$ is of the form $g + H$ where $H$ is a semi-abelian subvariety of $G$ and $g \in G$; furthermore, since $\varphi_{|V}$ has infinite order, $H$ is positive-dimensional. If $g+H= V$ then there is nothing else to prove. Otherwise, consider the quotient $\pi:G \to G/H$. From the description of the morphisms of $G$ provided by Lemma \ref{L:endomorphisms}, it is clear that there exists a morphism $\phi : G/H \to G/H$ such that $\pi\circ \varphi = \phi \circ \pi$. Since the image $\pi(V)$ is invariant by $\phi$ and $\phi$  has Zariski dense orbits, the result follows by induction on the dimension of $G$.
\end{proof}

\subsection{Quasi-albanese morphism of manifolds with Zariski dense dynamics}
For a detailed introduction to Albanese varieties and quasi-Albanese maps see e.g. \cite{Fujino2015ONQM}.

If $X$ is projective manifold and $D$ is a simple normal crossing divisor on $X$ then the Albanese variety of $(X,D)$ is the
semi-abelian variety biholomorphic to the complex abelian Lie group
\[
    \frac{H^0(X, \Omega^1_X(\log D))^*}{H_1(X-D, \mathbb Z)/\mathrm{Tor}}.
\]
If $x_0 \in X-D$ is an arbitrary fixed point then the holomorphic map
\begin{align*}
    \alb_{(X,D)} : X- D &\longrightarrow \Alb(X,D) \\
    x&\longmapsto \left\{ \omega \mapsto \int_{x_0}^ x \omega \right\}
\end{align*}
is actually a morphism of quasi-projective varieties called the quasi-Albanese
map of $(X,D)$. It is characterized, up to translations of $\Alb(X,D)$, by the following universal
property: for any morphism $h : X-D \to A$ to a semi-abelian variety $A$, there exists a
 morphism  $g: \Alb(X,D) \to A$ making the  diagram
\[
\begin{tikzcd}
    X-D \arrow[rr,"\alb_{(X,D)}"] \arrow[drr,"h"] & &  \Alb(X,D) \arrow[d,"g"] \\
     && A
\end{tikzcd}
\]
commutative.

\begin{lemma}\label{L:induz qalbaneseX}
    Let $X$ be a projective manifold and $f: X \dashrightarrow X$ be a dominant rational map.
    If $D$ is a simple normal crossing divisor such that $f^*D$ has support contained in the support
    of $D$ then $f$ induces  an
    endomorphism $f_*$ of $\Alb(X,D)$ such that the following  diagram
    \[
    \begin{tikzcd}
        X-D \arrow[rr,dashed,"f"] \arrow[d,"\alb_{(X,D)}"] & &  X - D \arrow[d,"\alb_{(X,D)}"] \\
        \Alb(X,D) \arrow[rr,"f_*"] && \Alb(X,D)
    \end{tikzcd}
    \]
    commutes.
\end{lemma}
\begin{proof}
    Since, by assumption, the support of $f^*D$ is contained in
    the support of $D$ we have an induced morphism $f^* : H^0(X, \Omega^1_X(\log D)) \to H^0(X,\Omega^1_X(\log D))$. Dualizing this morphism, we get a morphism $f^{**} : H^0(X, \Omega^1_X(\log D))^* \to H^0(X,\Omega^1_X(\log D))^*$, which descends to a morphism of the quotient $\Alb(X,D) = H^0(X,\Omega^1_X(\log D)^*/H_1(X-D,\mathbb Z)$ which we still denote by $f^{**}$. Fix $x_0 \in X-D$ and  choose the Albanese morphism in such a way that $\alb_{(X,D)} (x_0) = 0$. If we set
    \begin{align*}
        f_* : \Alb(X,D) & \longrightarrow \Alb(X,D) \\
            a & \mapsto f^{**}(a) + \alb_{(X,D)}(f(x_0))
    \end{align*}
    then  we get a morphism $f_*$ with the sought properties.
\end{proof}

\begin{lemma}\label{L:qalbanese dominates}
    Let $f: X \dashrightarrow X$ be a rational map on a projective manifold $X$ with a Zariski dense orbit and let $D$ be a simple normal crossing
    divisor on $X$. If $f^*D$ has support contained in the support of $D$ then  the quasi-Albanese morphism $\alb_{(X,D)} : X- D \to \Alb(X,D)$ is a  dominant  rational map.
\end{lemma}
\begin{proof}
    Lemma \ref{L:induz qalbaneseX} guarantees the existence of $f_* : \Alb(X,D) \to \Alb(X,D)$ such that $\alb_{(X,D)} \circ f = f_* \circ \alb_{(X,D)}$. Since $f$ has a Zariski dense orbit, the same holds true for the restriction of $f_*$ to the closure of the image of $\alb_{(X,D)}$. Proposition \ref{P:zariski closure} implies that the closure of the image of $\alb_{(X,D)}$ is equal to a semi-abelian subvariety of $\Alb(X,D)$. The universal property of the Albanese map implies that this semi-abelian subvariety must coincide with $\Alb(X,D)$, showing that $\alb_{(X,D)}$ is dominant.
\end{proof}

\subsection{Proof of Theorem \ref{THM:qalbanese kod0}} \label{SS:proofTHC}
    Assume  first that $\alb_{(X,D)}$ is generically finite, i.e., its general fiber consists of finitely many points.

    Let $n = \dim X = \dim \Alb(X,D)$. Let $\omega_1, \ldots, \omega_n$ be a basis of $H^0(X,\Omega^1_X(\log D))$ and consider the logarithmic $n$-form $\Omega = \omega_1 \wedge  \ldots \wedge \omega_n \in H^0(X,\Omega^n_X(\log D))$.
    Denote by $E$ the zero divisor of $\Omega$. We claim that no positive multiple of $E$ moves in a linear system. In other words, since  $E$ is linearly equivalent to $K_X + D$,  the logarithmic Kodaira dimension of $(X,D)$ is equal to zero. Aiming at a contradiction, assume that some positive multiple of $E$ moves in a linear system.

    Let $d$ be the degree of the field extension $\alb_{(X,D)}^*(\mathbb C( \Alb(X,D)) \subset \mathbb C(X)$, i.e. the cardinality of a general fiber of $\alb_{(X,D)}$. We want to prove that $d =1$. Assume that $d>1$. For $i \in \{ 1,2,\ldots, d, \infty\}$, set
    $
        \Sigma_i = \{ x \in \Alb(X,D) \, | \, \textrm{the cardinality of} \alb_{(X,D)}^{-1}(x) = i \}.
    $
    Note that the restriction of  $E$ to $X - D$ coincides with the ramification divisor of the quasi-Albanese morphism
    $\alb_{(X,D)} : X- D \to \Alb(X,D)$. Its image under the quasi-Albanese morphism coincides with
    \[
        \Alb(X,D) - \Sigma_d = \Sigma_{\infty} \cup \bigcup_{i=1}^{d-1} \Sigma_i \, .
    \]
    The set $\Sigma_{\infty}$ is closed but not necessarily $f_*$ invariant. Nevertheless, a simple dimension count shows that $\dim \Sigma_{\infty} \le \dim \Alb(X,D) - 2$. The set $\Sigma = \cup_{i=1}^{d-1} \Sigma_i$ is not necessarily closed, but its closure $\overline \Sigma$ is $f_*$-invariant since $f_* \circ \alb_{(X,D)}  = \alb_{(X,D)} \circ f$.

    Since we are assuming that $mE$ moves in a linear system, it cannot be contracted by the (generically finite) morphism $\alb_{(X,D)}$. Hence $\overline \Sigma$ must have irreducible components of codimension one. After replacing
    $f$ (and hence $f_*$) by a suitable iterate, we can assume the existence of an irreducible hypersurface $V$  of $\Alb(X,D)$ contained in $\overline \Sigma$ and invariant by $f_*$. Corollary \ref{C:invariant hypersurface} guarantees the existence of semi-abelian subvariety $B \subset \Alb(X,D)$ such that $V$ equals to a translate of $B$. Let $C$ be the quotient $\Alb(X,D)/B$ and  consider the composition
    \[
        X \dashrightarrow \Alb(X,D) \to C \, .
    \]
    Since $\dim C=1$, $C$ is equal to an elliptic curve or $\mathbb C^*$. To achieve a contradiction, we will show that the endomorphism $\varphi:C \to C$ induced by $f_*$ is of finite order.

    Let $\eta$ be the pull-back to $X$ of a non-zero holomorphic $1$-form on  $C$ having a pole of order one at infinity when $C=\C^*$. Observe that  $\eta$ vanishes on a hypersurface $W_0 \subset W=\alb_{(X,D)}^{-1}(V)\cap E$ which dominates $V$. Indeed, by  choice, the quasi-Albanese map does not     contract (all the irreducible components of) $W$ but, at the same time, $W$ is contained in $E$ the ramification divisor of $\alb_{(X,D)}$. These two observations suffice to guarantee  the existence of $W_0$. Notice also that $\eta$ defines an algebraically integrable foliation invariant by $f$. Moreover, since the quasi-Albanese morphism is dominant, $f^* \eta$ is a constant multiple of $\eta$. Therefore the set $\cup_{n\in \mathbb N} f^{-n}(W_0)$ is contained in the support of the zero divisor of $\eta$ and must be a finite union of hypersurfaces. It follows that $\varphi : C \to C$ has a fixed point with finite backward orbit. Since $C$ is either an elliptic curve or $\mathbb C^*$, this implies that $\varphi$ is a finite order automorphism contradicting the existence of a Zariski dense orbit of $f$. This establishes our claim that the  logarithmic Kodaira dimension of $(X,D)$ is equal to zero.

    Since $(X,D)$ has logarithmic Kodaira dimension zero and its quasi-Albanese morphism is dominant, we can apply \cite[Corollary 29]{MR622451} to conclude that the quasi-Albanese morphism is a birational map as claimed.

    Assume now that the general fiber of $\alb_{(X,D)}$ is positive dimensional. Let $\mathcal G$ be the foliation defined by the fibers of $\alb_{(X,D)}$ and let $g : X \dashrightarrow Y$ be a rational map  with irreducible general  fiber to a projective manifold $Y$,  defining $\G$, and such that the image of $D$ is a simple normal crossing divisor $D_Y$. Notice that $g: X \dashrightarrow Y$ is one Stein factorization of $\alb_{(X,D)}: X -D \to \Alb(X,D)$. By construction, every logarithmic $1$-form $\omega \in H^0(X,\Omega^1_X(\log D))$ is the pull-back of logarithmic $1$-form in $H^0(Y,\Omega^1_Y(\log D_Y))$. In particular,
    \begin{equation}\label{E:cota boba}
        \dim \Alb(Y,D_Y) \ge \dim \Alb(X,D) = \dim Y.
    \end{equation}

    Lemma \ref{L:induz qalbaneseX} implies that the rational map $f:X \dashrightarrow X$  preserves
    the foliation $\G$ and therefore induces a rational map $f_Y : Y \dashrightarrow Y$. Since $f$ has a Zariski dense orbit, so does $f_Y$.  Moreover, the support of $f_Y^*(D_Y)$ is contained in the support $D_Y$.

    Consider the commutative diagram induced by the rational map $g: X \dashrightarrow Y$.
    \[
        \begin{tikzcd}
            X-D \arrow[d,dashed,"g"] \arrow[rr,"\alb_{(X,D)}"] \arrow[drr,dashed] & &  \Alb(X,D) \arrow[d,"g_*"] \\
            Y - D_Y \arrow[rr,"\alb_{(Y,D_Y)}"] && \Alb(Y,D_Y)
        \end{tikzcd}
    \]
    Since $f_Y$ has Zariski-dense orbits, Lemma \ref{L:qalbanese dominates} combined with the inequality (\ref{E:cota boba})   implies that the quasi-Albanese morphism $\alb_{(Y,D_Y)}$ is surjective and generically finite. As before, we deduce that $\alb_{(Y,D_Y)}$ is a birational morphism. Since $g$ has
    irreducible general fiber, so does the diagonal arrow in the diagram above. Its  commutativity  implies that $\alb_{(X,D)}$ also has irreducible general fiber as claimed.
\qed

\section{Symmetries of virtually transversely additive foliations}\label{S:Zariski dense}

\subsection{Reduction to the transversely additive case}

\begin{lemma}\label{L:covering}
    Let $\F$ be a virtually transversely additive foliation on a projective manifold $X$ and $f:X \dashrightarrow X$ be
    a rational map such that $f^*\F = \F$. If $\F$ is not transversely additive then there exists a projective manifold $Y$, a transversely additive foliation
    $\G$ on $Y$, an automorphism $\varphi : Y \to Y$ of finite order, a $\varphi$-equivariant morphism $\pi : Y \to X$, and a rational map $g:Y \dashrightarrow Y$ such that
    \[
        \G = \varphi^* \G = g^* \G = \pi^* \F
    \]
    and the diagram
    \[
    \begin{tikzcd}
        Y \arrow[rr,dashed,"g"] \arrow[d,"\pi"] & &  Y\arrow[d,"\pi"] \\
        X \arrow[rr,dashed, "f"] && X \\
    \end{tikzcd}
    \]
    commutes.
\end{lemma}
\begin{proof}
    Let $\omega \in H^0(X,\Omega^1_X\otimes N_{\F})$.
    Since $\F$ is virtually transversely additive, but not transversely additive, there exists a unique flat logarithmic connection on the normal bundle $\F$ such that
    $\nabla(\omega)=0$.

    Let $N$ be the total space of the normal bundle of $\F$ and denote by $p:N \to X$ the natural projection.
    The flat sections of the connection $\nabla$ define an algebraically integrable foliation $\mathcal H$ on $N$. Over the complement of the polar locus of $\nabla$ the leaves of $\mathcal H$ are \'etale coverings of $X$. The action $\psi : \mathbb C^* \times N \to N$ of $\mathbb C^*$ on $N$ by fiberwise multiplication preserves the foliation $\mathcal H$.

    The derivative of $f$ defines  $\hat f : N \dashrightarrow N$, a lift of $f$ to $N$. Since the flat connection  $\nabla$ is unique, $\hat f$ preserves the foliation $\mathcal H$. If $L$ is a leaf of $\mathcal H$ which dominates $X$ then $\hat f(L)$ is also a leaf of $\mathcal H$ which dominates $X$. Thus there exists $\lambda \in \mathbb C^*$ such that $\psi(\lambda, \hat{f} (L)) = L$. Moreover,
    the subgroup $\{ \mu \in \mathbb C^* ; \psi(\mu,L)= L \}$
    is a cyclic subgroup generated by $\xi \in \mathbb C^*$,  a primitive root of the unity. The quotient of $L$ by this subgroup is birational to $X$.

    To conclude the proof of the lemma, it suffices to take $Y$ equal to an equivariant (with respect to the action of $\psi(\xi, \cdot): L \to L$) resolution of singularities of the Zariski closure of $L$; $\varphi : Y \to Y$ equal to the automorphism of $Y$ induced by $\psi(\xi, \cdot)$; $\pi: Y \to X$ equal to the morphism induced by the restriction of $p:N \to X$ to $L$; $\mathcal G$ equal to the foliation on $Y$ induced by the restriction of $p^*\mathcal F$ to $L$; and $g : Y \dashrightarrow Y$ equal to the rational map induced by the  restriction of $\psi(\lambda, \hat{f} (\cdot))$ to $L$.
\end{proof}

\subsection{Reduction of singularities for transversely additive foliations}
The definition below is due to Cano-Cerveau, see for instance \cite{MR1760842}.

\begin{dfn}[Simple singularities]\label{D:simple}
    Let $\F$ be a germ of codimension one foliation on $(\mathbb C^n,0)$. The foliation
    $\F$ has simple singularities if there exists formal coordinates $x_1,\ldots, x_n$ and an integer
    $r$, $2 \le r \le n$, (the dimension type of $\F$) such that
    $\F$ is defined by a differential form $\omega$ of one of the following types:
    \begin{enumerate}
        \item\label{I:CC1}  There are complex numbers $\lambda_i \in \mathbb C^*$ such that
        \[
            \omega = \sum_{i=1}^r \lambda_i \frac{dx_i}{x_i} \, ,
        \]
        and $\sum_{i=1}^r a_i \lambda_i = 0$ for non-negative integers $a_i$ implies $(a_1, \ldots, a_r)=0$.
        \item\label{I:CC2} There exist an integer $1\leq k \le r$,  positive integers $p_1, \ldots, p_k$ , complex numbers $\lambda_2, \ldots,\lambda_r$, and a formal power series $\psi \in t \cdot \mathbb C[[t]]$  such that
        \[
            \omega =  \sum_{i=1}^k p_i \frac{dx_i}{x_i} + \psi(x_1^{p_1}\cdots x_k^{p_k}) \sum_{i=2}^r \lambda_i \frac{dx_i}{x_i}
        \]
        and $\sum_{i=k+1}^r a_i \lambda_i = 0$ for non-negative integers $a_i$ implies $(a_{k+1}, \ldots, a_r)=0$.
    \end{enumerate}
\end{dfn}

The non-resonance condition ( $\sum a_i \lambda_i =0 \implies (a_1, \ldots, a_r) =0$ or $(a_{k+1}, \ldots, a_r)=0$)  present in both items of the definition above implies that any irreducible component of the  exceptional divisor of any birational morphism with center contained in  $\sing(\F)$  is invariant by the transformed foliation, cf. \cite[Proposition 15]{MR1677402}. In other words, simple singularities are non-dicritical singularities.

\begin{rem}\label{R:misterious index}
    The second summation in Item (\ref{I:CC2}) of the definition may cause some discomfort at first sight.
    However, if one allows the second summation to start at $1$, and keep all the other conditions, one gets the same definition. Indeed,  one can rewrite
    \[
      \sum_{i=1}^k p_i \frac{dx_i}{x_i} + \psi(x_1^{p_1}\cdots x_k^{p_k}) \sum_{i=1}^r \lambda_i \frac{dx_i}{x_i}
    \]
    as
    \begin{align*}
       & \sum_{i=1}^k p_i \frac{dx_i}{x_i} + \psi(x_1^{p_1}\cdots x_k^{p_k}) \left( \sum_{i=2}^k \left( \lambda_i -  \frac{\lambda_1 p_i}{p_1} \right) \frac{dx_i}{x_i}  \right) \\
       + \, \, &\lambda_1 \psi(x_1^{p_1}\cdots x_k^{p_k}) d \log\left(\frac{1}{x_1^{p_1} \cdots x_k^{p_k}} \right) \, .
    \end{align*}
    Since the differential form on the second line has no poles and is closed, one can make a change of coordinates of the form $(x_1, \ldots, x_n) \mapsto (ux_1,\ldots,x_n)$, where $u$ is a suitable unit, to make it disappear.
\end{rem}

\begin{prop}\label{P:Seidenberg for closed forms}
    Let $\mathcal F$ be a foliation on a projective manifold $X$ defined by a
    a closed rational $1$-form $\omega$. Then there exists a birational morphism $\pi: Y \to X$
    from a projective manifold $Y$ to $X$ such that $\pi^* \mathcal F$ has simple
    singularities.
\end{prop}
\begin{proof}
    Let $D$ be the polar divisor of $\omega$ and consider a log resolution $p_1 : X_1 \to X$  of
    $(X,D)$. Let $\omega_1=p_1^*\omega$ be the transformed $1$-form and denote by $D_1$ it polar divisor.

    Let $x$ be an arbitrary closed point in the support of $D_1$. In a sufficiently small
    neighborhood of $x$ choose coordinates $(x_1, \ldots, x_n)$   such that $D_1 \subset\{ x_1 \cdots x_k = 0\}$.  The $1$-form $\omega_1$ can be written as
    \[
        \sum_{i=1}^k \lambda_i \frac{dx_i}{x_i} + d \left( \frac{g}{x_1^{m_1}\cdots x_k^{m_k}} \right) \, .
    \]
    The meromorphic function $\frac{g}{x_1^{m_1}\cdots x_n^{m_n}}$ is not intrinsically associated to the situation
    since we can add constants to it, and we can also choose another system of coordinates. Nevertheless, its  base ideal, i.e. the ideal generated by $g$ and $x_1^{m_1}\cdots x_k^{m_k}$, does not depend on the choices made. Therefore, we
    have a globally defined base ideal $\mathcal I$.

    Let $p_2: X_2 \to X_1$ be a resolution of $\mathcal I$ and set $\omega_2 = p_2^* \omega_1$. Now, at a neighborhood of any point of $D_2$,  there exists
    coordinates $(x_1, \ldots, x_n)$ such that
    \[
        \omega_2 = \sum_{i=1}^k \lambda_i \frac{dx_i}{x_i} + d \left( \frac{1}{x_1^{m_1}\cdots x_k^{m_k}} \right) \, .
    \]
    From these explicit formulas, one sees that all the singularities of the foliation defined by $\omega_2$ consist of the singularities of the polar divisor of $\omega$ union a closed set disjoint from the support of the polar divisor of $\omega_2$. Moreover, taking into account  Remark \ref{R:misterious index}, one sees that the singularities of the foliation located at the singularities of $D$ are almost simple: they satisfy all the conditions of Definition \ref{D:simple} except, perhaps, the non-resonance condition.

    We can apply \cite[Theorem 2]{MR3302577} to produce a birational morphism $p_3 : X_3 \to X_2$ such that all the singularities of  $\omega_3 = p_3^* \omega_2$ located at a neighborhood of the polar divisor of $\omega_3$ are simple.
    If the foliation defined by $\omega_3$ has non-simple singularities then they are disjoint from the polar divisor of $\omega_3$. Let $\Sigma$ be one connected component of the singular set of $\omega_3$ disjoint from its polar divisor.  Since $\omega_3$ is closed, for any point of $\Sigma$ the form $\omega_3$ is locally exact.
    Moreover, we can choose unique primitives for $\omega_3$ by imposing that they are constant along $\Sigma$. To wit, there exists an open neighborhood $U_{\Sigma}$ of $\Sigma$ and a holomorphic function $f_{\Sigma}: U_{\Sigma} \to \mathbb C$
    such that the restriction $\omega_3$ to $U_{\Sigma}$ is equal to $df_{\Sigma}$. Resolution of singularites
    of the hypersurfaces $f_{\Sigma}^{-1}(0)$ with $\Sigma$ ranging over the irreducible components of the singularities
    of the foliation defined by $\omega_3$ disjoint from $(\omega_3)_{\infty}$ provides a morphism $p_4: X_4 \to X_3$ such that the foliation defined by $p_4^* \omega_3$ has only simple singularities.

    The result follows by taking $Y=X_4$ and $\pi = p_1 \circ p_2 \circ p_3 \circ p_4$.
\end{proof}

\subsection{Foliations defined by logarithmic $1$-forms}

\begin{prop}\label{P:f preserves log}
    Let $X$ be a projective manifold and $f: X \dashrightarrow X$ be a dominant rational map preserving a transversely additive codimension one foliation
    $\mathcal F$ defined by a closed logarithmic $1$-form $\omega$. If the very general orbit of $f$ is Zariski dense and
     $\F$ is purely transcendental then there exists
     \begin{enumerate}
        \item a semi-abelian variety $A$;
        \item a birational map $\varphi : X \dashrightarrow A$;
        \item an endomorphism $g : A \to A$; and
        \item a foliation $\G$ on A defined by a Lie algebra of invariant vector fields
     \end{enumerate}
     such that $\F = \varphi^* \G$ and $ \varphi \circ f = g \circ \varphi$.
\end{prop}
\begin{proof}
    The invariance of $\F$  by $f$ implies the existence of a rational function $h \in \C(X)$ such that $f^* \omega = h \omega$. Since $\omega$ is closed, differentiation of this identity implies that $ d h \wedge \omega =0$, i.e. $h$ is a rational first integral of $\F$. By assumption, $h$ must be constant.

    Proposition \ref{P:Seidenberg for closed forms} allows us to assume, after replacing $X$ by a suitable birational model, that $\mathcal F$ has only simple singularities. In particular, we can assume that $\F$ is a non-dicritical foliation. Further blow-ups centered at the singular set of the polar divisor (which is contained in the singular set of $\F$) allow us to assume that the polar divisor is a simple normal crossing divisor and that the foliation is still non-dicritical.

    Consider  $f^*D$, the pull-back of the polar divisor of $\omega$. Since $\mathcal F$ is non-dicritical, every irreducible component of $f^*D$ must be an irreducible component of the polar divisor of $f^*\omega$. Indeed, if $H$ is an irreducible component of the support of $f^*D$ then $H$ is mapped to
    the support of $D$. If $f(H)$ is not contained in the singular set of $\F$  then, since $f(H)$ is contained in the polar locus of $\omega$, $H$ is clearly contained in the polar locus of $f^*\omega$. If instead $f(H)$ is contained in $\sing(\F)$, let  $\rho_0: X_1 \to X$ be the blow-up of $X$ along $f(H)$ and consider the rational map $f_1 : X \dashrightarrow X_1$ defined as $f_1 = \rho_0^{-1} \circ f$. Set $\F_1=\rho_0^* \F$.
    Since $\F$ is non-dicritical, the exceptional divisor of $\rho_0$ is $\F_1$-invariant and, therefore, contained in the polar locus of $\omega_1=\rho_0^*\omega$. If  $f_1(H)$ is not contained in the singular set of $\F_1$ then we conclude as before. If not, we let $\rho_1:X_2 \to X_1$ be the blow-up of $X_1$ along $f_1(H)$ and consider the rational map $f_2 : X \dashrightarrow X_2$  given by $f_2 = \rho_1^{-1} \circ f_1$  and we repeat the argument above with the foliation $\F_2= \rho_1^* \F_1$. Again, the non-dicriticalness of $\F$ implies that the exceptional divisor of $\rho_1$ is $\F_2$-invariant and contained in the polar locus of $\omega_2 = \rho_1^*\omega_1$. Proceeding inductively,
    we reach a situation where $f_i(H)$ is not contained in the singular set of $\F_i$ because, according to \cite[Theorem VI.1.3]{MR1440180}, there will exist an $i$ such that the dimension of  $f_i(H)$ is equal to the dimension of $H$. In this case, $f_i(H)$ is not contained in $\sing(\F_i)$, since the singular set of $\F_i$ has codimension at least two.

    Since $f^*\omega$ and $\omega$ differ by a multiplicative constant we deduce that the support of $f^*D$ is contained in the support of $D$. Notice also, that $\omega$ is the pull-back under the quasi-Albanese map of $(X,D)$ of a $1$-form $\eta$ on $\Alb(X,D)$. In particular, the fibers of $\alb_{(X,D)}$ are tangent to $\F$.  By assumption, $\F$ is purely transcendental and we deduce that $\alb_{(X,D)}$ is generically finite. Theorem \ref{THM:qalbanese kod0} implies that $\alb_{(X,D)}$ is a surjective birational map. The proposition follows by taking $A=\Alb(X,D)$, $\varphi= \alb_{(X,D)}$, $g = f_*$ and $\G$ equal to the foliation defined by $\eta$.
\end{proof}

\subsection{Foliations defined by closed rational $1$-forms}\label{SS:closed}

Let $\omega$ be a closed rational $1$-form on a projective manifold $X$. Assume that the polar divisor of $\omega$ is
supported on a simple normal crossing divisor $D$. Hodge theory implies that the $H^1(X-D, \C)$ is the direct sum
\[
    H^0(X, \Omega^1_X(\log D)) \oplus \overline{ H^0(X,\Omega^1_X)} \, .
\]
Therefore, we can decompose $\omega$ as $\omega_{\log} + \omega_{\II}$, where $\omega_{\log} \in H^0(X, \Omega^1_X(\log D))$ is a  logarithmic $1$-form  and $\omega_{\II}$ is  a closed $1$-form without residues (in the classical literature these are called differentials of the second kind) and with de Rham cohomology class $[\omega_{\II}]$ in $\overline{H^0(X,\Omega^1_X)} \subset H^1(X-D,\C)$.

From now on, we will assume that  $\omega_{\II}$ is non zero. Therefore, it (or rather its class in $H^1(X,\mathcal O_X)$) determines a rank one vectorial extension of $\Alb(X,D)$ which we will call $B$. Concretely, $B$ can be seen as the quotient
\[
    \frac{\left( H^0(X,\Omega^1_X(\log D)) \oplus \mathbb C \omega_{\II} \right)^* }{H_1(X-D,\mathbb Z)/Tor} \, .
\]
Analogously to the quasi-Albanese variety, $B$ is the target of a natural rational map (well-defined up to translations)
\begin{align*}
    \beta : X - D & \longrightarrow B \\
    x & \mapsto \left\lbrace \alpha \mapsto \int_{x_0}^x \alpha \right\rbrace
\end{align*}
that factors the quasi-Albanese morphism as depicted below
\[
    \begin{tikzcd}
        X-D \arrow[rr,"\beta"] \arrow[drr,"\alb_{(X,D)}" below] & & B \arrow[d] \\
        && \Alb(X,D) .\\
    \end{tikzcd}
\]

Suppose now that the foliation $\F$ defined by $\omega$ is not algebraically integrable and  has simple singularities (in particular, its singularities are non-dicritical). If $\F$ is preserved by rational map $f: X \dashrightarrow X$ then, as argued at the beginning of the proof of Proposition \ref{P:f preserves log}, $f^* \omega = \lambda \omega$, for some $\lambda \in \mathbb C^*$. Since the decomposition $\omega = \omega_{\log}  + \omega_{\II}$ is canonical, we have the identities $f^* \omega_{\log} = \lambda \omega_{\log}$ and $f^*\omega_{\II} = \lambda \omega_{\II}$. Furthermore, the support of $f^* D$ is contained in the support of $D$. Arguing as in the proof of Lemma \ref{L:induz qalbaneseX} we deduce the existence of  a morphism $\hat{f_*}: B \to B$ which fits into the commutative  diagram
\[
    \begin{tikzcd}
        X-D \arrow[rr,dashed,"f"] \arrow[d, "\beta"] \arrow[dd, bend right=40, "\alb_{(X,D)}" left] & &  X - D  \arrow[d, "\beta" left]  \arrow[dd,bend left, "\alb_{(X,D)}"] \\
        B  \arrow[d] \arrow[rr,"\hat{f_*}"] && B \arrow[d]\\
        \Alb(X,D) \arrow[rr,"f_*"] && \Alb(X,D) .
    \end{tikzcd}
\]

\begin{lemma}\label{L:conclusion}
    Notations as above. Assume that $\beta$ is a dominant morphism.  If the foliation $\F$ is purely transcendental and the rational map $f : X \dashrightarrow X$ has a Zariski dense orbit then the morphism $\beta : X - D \to B$ is generically finite. Moreover, if $\beta$ is not birational then there exists a rational function $h \in \mathbb C(X)$ such that $\omega_{\II} = dh$ and $f^* h = \lambda h$.

\end{lemma}
\begin{proof}
    Let $F$ be a general fiber of the quasi-Albanese morphism $\alb_{(X,D)}$. We know from Theorem \ref{THM:qalbanese kod0} that $F$ is irreducible. If the dimension of $F$ is zero then Theorem \ref{THM:qalbanese kod0} implies that $\alb_{(X,D)}$ is birational. Since  $\beta$ factors $\alb_{(X,D)}$ it must also be a birational morphism over its image and there is nothing else to prove.

    From now on let us assume that $\dim F \ge 1$. We claim that $\dim F=1$. Indeed, if $i: F \to X-D$ denotes the inclusion then $i^*\omega_{\log}=0$ %while $i^*\omega_{\II}= dh$ for some $h \in \mathcal O_X(F) \cap \mathbb C(\overline F)$
    while $i^*\omega_{\II}= dh$ for some $h \in  \mathbb C(\overline F)$ ( $\overline F$ stands for the Zariski closure of $F$ in $X$). Note that  $h$ is necessarily non constant, because $\F$ is purely transcendental.
    %where $\overline F$ denotes the Zariski closure of $F$ in $X$.
    If $\dim F\ge 2$ then $h$ defines an algebraically integrable subfoliation of $\mathcal F_{|\overline{F}}$. Since $F$ is general this implies that $\F$ is not purely transcendental, again contrary to our assumptions.

    To deduce that $\beta$ is generically finite, it suffices to notice that $\beta$ maps $F$ to a fiber of $B \to \Alb(X,D)$ and $\beta_{|F}$ is essentially given by the rational function $h$ up to an additive constant. To conclude observe as before that $h$ is not constant as otherwise $F$ would be contained in a leaf of $\F$ and $\F$ would  not be purely transcendental.

    From now on, assume that $\beta$ is not birational. It remains to show that $\omega_{\II}$ is exact.  Let $C \subset B$ be the Zariski closure of the codimension one part of the critical locus of the morphism $\beta$. As we are assuming that $\beta$ is not birational, and since $\alb_{(X,D)}$ has connected general fiber according to Theorem \ref{THM:qalbanese kod0}, $C$ is non-empty and has an irreducible component $C_0$ which dominates $\Alb(X,D)$. Moreover, $C_0$ must be invariant by a power of the morphism $\hat{f_*}$.

    The morphism $f_* : \Alb(X,D) \to \Alb(X,D)$ can be decomposed as the sum of an
    endomorphism $\varphi$ of the abelian group $\Alb(X,D)$ and a translation. After replacing $f$ by a suitable power, we may assume that all roots of the minimal polynomial of $\varphi \in \End(\Alb(X,D))$ different from $1$ are not roots of the unity. It follows that $\Alb(X,D)$ is isogeneous to a product of quasi-abelian varieties $A_1\subset \Alb(X,D)$ and $A_2\subset \Alb(X,D)$ such that $f_*$ induces on $A_1$ a translation and on $A_2$ an endomorphism with Zariski dense set of periodic orbits, see \cite[proof of Theorem 1.1]{MR3937327}.

    If $\lambda=1$ then $\omega_{log}$ is the pull-back of a $1$-form under the morphism $\Alb(X,D) \to \Alb(X,D)/A_2$. Moreover,  $\omega_{\II}$ is the pull-back of a $1$-form under a morphism $B\to B_1$ where $B_1$ is a rank vectorial extension of $\Alb(X,D)/A_2$. Since we are assuming that $\omega_{log} + \omega_{\II}$ defines a purely transcendental foliation, we have that $A_2$ must be trivial,  $A_1 = \Alb(X,D)$ and $B_1 = B$. Thus if $\lambda =1$ then (some power of $f_*$) is a translation and the same holds true for $\hat{f_*}$. Since $\hat{f_*}$ is a translation with  Zariski dense orbits, then every orbit of $\hat{f_*}$ is Zariski dense contradicting the existence of $C_0$.

    If $\lambda\neq 1$ then we can argue in the same way to deduce that $A_2 = \Alb(X,D)$ and that $f_*$ is (conjugated by a translation to) an endomorphism of $\Alb(X,D)$ with a Zariski dense set of periodic orbits. The action of $\hat{f_*}^n$ on fibers of $B \to \Alb(X,D)$ over periodic points is given by $z \mapsto \lambda^n z + b_n$ for some $b_n \in \mathbb C$. It follows that $C$ must intersect such fiber in
    the unique fixed point of $z \mapsto \lambda^n z + b_n$. Thus  $C_0$  is a section of the projection $B \to \Alb(X,D)$ invariant by $\hat{f_*}$. The existence of a section implies that the vectorial extension $B$ defined by $\omega_{\II}$ is trivial. Hence the class of  $\omega_{\II}$  in  $\overline{H^0(X,\Omega^1_X)}$ %H^1(X,\mathcal O_X)$
    is zero, showing that $\omega_{\II}$ is exact. If $g \in \mathbb C(X)$ is such that $\omega_{\II} = dg$ then $f^*g = \lambda g + \mu$ for some $\mu \in \mathbb C$. To conclude  it suffices to take $h = g + \mu/\lambda$.
\end{proof}

\subsection{Proof of Theorem \ref{THM:core of structure}}
\label{sec: Proof B}
Let $\F$ be a purely transcendental virtually transversely additive foliation on a projective manifold $X$ invariant by
a rational map $f:X \dashrightarrow X$ with Zariski dense orbits.

Lemma \ref{L:covering} implies that $(X,\F,f)$ is birationally equivalent to the quotient, by a finite cyclic group, of a foliation $\G$ defined by a closed rational $1$-form $\omega$ on a projective manifold $Y$.

If $\omega$ is logarithmic then we can apply Proposition \ref{P:f preserves log} to conclude that $Y$ is birationally equivalent to a semi-abelian variety $A$ and that $f$ is conjugated to an endomorphism of $A$. This proves Theorem \ref{THM:core of structure} when $\omega$ is logarithmic.

If $\omega$ is not logarithmic then, thanks to Proposition \ref{P:Seidenberg for closed forms}, we can assume without loss of generality that  $\G$ has simple singularities and $D$, the polar divisor  of $\omega$, is simple normal crossing. If the quasi-Albanese morphism $\alb_{(Y,D)}$ is generically finite then the result follows from Theorem \ref{THM:qalbanese kod0}.

If $\alb_{(Y,D)}$ is not generically finite,  consider the morphism $\beta : Y - D \to B$ constructed in Subsection \ref{SS:closed}. Notice that $B$ is  a commutative algebraic group since it is constructed as a vectorial extension of $\Alb(Y,D)$. If $\beta$ is birational then there is nothing else to prove. If $\beta$ is not birational then Lemma \ref{L:conclusion} implies the existence of rational function $h \in \C(Y)$ (a primitive of $\omega_{\II}$) such that $g^* h = \lambda  h$ for some $\lambda \in \mathbb C^*$. Let $\eta$ be the logarithmic differential of $h$, i.e. $\eta =  \frac{dh}{h}$. If we set $\hat D$ as the reduced divisor with support equal to the union of the support of $D$ and the support of the polar divisor of  $\eta$ then  $g^* \hat D$ has support contained in $\hat D$. Moreover, the quasi-Albanese morphism of $(Y,\hat D)$ is generically finite since, by construction, it is non-constant on the generic fiber of $Y- D \to \Alb(X,D)$. We can apply Theorem \ref{THM:qalbanese kod0} to conclude that $Y$ is birationally equivalent to a semi-abelian variety.  \qed

\section{Foliations invariant by endomorphisms of projective spaces}\label{S:endo}

\subsection{Endomorphisms and their exceptional hypersurfaces}
Let $f:X \to Y$ be a finite morphism of $m$-dimensional compact complex manifolds. The number of preimages of
a generic point $y \in Y$ is called the topological degree of $f$ and it will be denoted by $\deg(f)$.

For a point $x\in X$ the local degree of $f$ at $x$  is defined as
\[
    \deg_x(f) = \max \left\{ {\rm Card} \left( f^{-1}(y) \cap B(x,\epsilon)  \right), y \in B(f(x),\epsilon), \epsilon \ll 1 \right\} \, ,
\]
and for an irreducible subvariety $Z \subset X$ the local degree of $f$ at $Z$ is defined as
\[
    \deg_Z(f) = \min \{ \deg_x(f), x \in Z\} \, .
\]

The branching divisor of $f$ is given by the formula
\[
    B(f) = \sum (\deg_H(f) - 1)\cdot H
\]
where the sum is taken over all irreducible hypersurfaces of $X$.
The branching divisor satisfies the relation
\begin{equation}\label{E:rcanonico}
    f^* K_Y \simeq K_X + B(f) \, ,
\end{equation}
where $\simeq$ denotes linear equivalence and $K_X$, respectively $K_Y$, are the canonical divisors of $X$, respectively $Y$.

From now on, let $f:X \to X$ be an  endomorphism;  then
\[
    f_* f^* = {\deg}(f) \cdot {\id} \, \qquad {\rm in}  \quad H^*(X,\mathbb Q) \, ,
\]
which implies that the irreducible subvarieties contracted by $f$ are rationally cohomologous to zero.
For K\"{a}hler manifolds it follows that every endomorphism is in fact a finite morphism.

The exceptional  hypersurface of $f$ is the largest reduced hypersurface $\mathcal E \subset X$ such that
$f^{-1}(\mathcal E) = \mathcal E = f(\mathcal E)$ (set theoretically), and  $\deg_{\mathcal E_i}(f) > 1$ for every irreducible component $\mathcal E_i$ of $\mathcal E$.

Note that after replacing $f$ by a suitable power $f^m$ we can suppose that $f^{-1}(\mathcal E_i) = \mathcal E_i $ for every
irreducible component $\mathcal E_i$ of the exceptional hypersurface.

From the definition of $\deg_H(f)$ it follows that $f^*\mathcal E_i = \deg_{\mathcal E_i}(f) \cdot \mathcal E_i$ for every
irreducible component $\mathcal E_i$ of the exceptional hypersurface. When $\mathcal E_i^{\dim X} \neq 0$ then the
identity $f^* {\mathcal E_i}^n = \deg(f) \mathcal E_i^n $ implies at once that
\[
   \deg_{\mathcal E_i}(f) ^{\dim X} = \deg(f) \, .
\]
More generally when $\mathcal E_{i_1}^{r_{i_1}} \cdots \mathcal E_{i_k} ^{r_{i_k}} \neq 0$, $r_{i_1}+\ldots + r_{i_k} = \dim X$, then
\begin{equation}\label{E:prodgrau}
  \prod_{j=1}^k \deg_{\mathcal E_{i_j}}(f)^{r_{i_j}}= \deg(f) \, .
\end{equation}

\begin{lemma}\label{L:log log}
    Let $X$ be a projective manifold with $\Pic(X) = \mathbb Z$ and let $f \colon X \to X$ be a (finite) endomorphism.
    If $\mathcal E$ is the exceptional hypersurface of $f$ then
    \begin{enumerate}
        \item\label{I:67} the pair $(X,\mathcal E)$ is log-canonical; and
        \item\label{I:68} every logarithmic form (of arbitrary degree)  with poles on $\mathcal E$ is closed.
    \end{enumerate}
\end{lemma}
\begin{proof}
    The hypothesis $\Pic(X) = \mathbb Z$ implies that the endomorphism $f$ is polarized. Item (\ref{I:67}) is then a direct consequence of \cite[Corollary 3.3]{MR3263168}.

    Item (\ref{I:68}) follows from \cite[Theorem 1.5 and Remark 1.5.2]{MR2854859}. Indeed, let $\omega \in H^0(X, \Omega^k_X(\log \mathcal E))$
    and let  $\pi : \tilde X \to X$ be a log resolution of the pair $(X,\mathcal E)$. If $U \subset X$ is the locus where $\mathcal E$
    is a normal crossing divisor then the differential form $\pi^*(\omega_{|U})$ extends to a logarithmic form $\tilde \omega$ on $\tilde X$. Deligne's Theorem implies that $\tilde \omega$ is closed and  the same holds true for $\omega$.
\end{proof}

\subsection{Pfaff Equations}
For us, a  codimension $p$ Pfaff equation $\mathcal P$ on a compact complex
manifold $X$ is a given by a line bundle $N_{\mathcal P}$ and an equivalence class of  twisted $p$-forms $[\omega] \in \mathbb P H^0(X, \Omega_X^p\otimes N_{\mathcal P})$ (not necessarily integrable nor decomposable) with zero sets of codimension at least two.

We will say that  an  irreducible hypersurface $H$ is invariant by $\mathcal P$ if for any local equation $h$ of $H$ and any local equation $\omega$ of $\mathcal P$ the $(p+1)$-form
\[
    \omega \wedge \frac{dh}{h}
\]
is holomorphic.

A meromorphic function $g \in \mathbb C(X)$ is a  first integral of $\mathcal P$ if for any local equation $\omega$ of $\mathcal P$
we have that
\[
   \omega \wedge dg \equiv 0 \, .
\]
The first integrals of $\mathcal P$ form a subfield of $\mathbb C(X)$ which we will denote $\mathbb C(\mathcal P)$. The
transcendence degree of $\mathbb C(\mathcal P)$ is bounded by the codimension of $\mathcal P$. Jouanolou-Ghys'  Theorem can be easily extended to Pfaff equations.

\begin{thm}[Jouanolou-Ghys]
    If $\mathcal P$ is a codimension $p$ Pfaff equation
    on a compact complex manifold then  $\mathbb C(\mathcal P) \neq \mathbb C$ if, and only if, $\mathcal P$ admits
    an infinite number of invariant hypersurfaces.
\end{thm}

Let $\mathcal U=\{ U_i\}$ be a suitable open covering of $X$ and let $\omega_i \in \Omega^p_X(U_i)$ be $p$-forms with zero
set of codimension at least two defining the restriction of  $\mathcal P$ to $U_i$.

If $f:X \to  X$ is an endomorphism then $f^* \mathcal P$ is the Pfaff equation on $X$ induced by the collection $f^* \omega_i/(f^*
\omega_i)_0$. A simple computation shows that
\[
    (f^* \omega_i)_0 = \sum (\deg_H(f) - 1 )\cdot H \,
\]
where  the sums are  taken over all $\mathcal P$-invariant irreducible hypersurfaces, cf. \cite[Chapter 2]{MR3328860}.
It follows that
\[
    f^* N_{\mathcal P}   =  N_{f^*{\mathcal P}}  \otimes \mathcal O_X\left(\sum (\deg_H(f) - 1 )\cdot H\right) \, ,
\]
We will say that $\mathcal P$ is $f$-invariant  when $f^* \mathcal P = \mathcal P$.

\subsection{Proof of Proposition \ref{PROP:endo}} Proposition \ref{PROP:endo} follows immediately from the more general result below.

\begin{prop}\label{P:endogeral}
    Let $X$ be a projective manifold with $\Pic(X)$ isomorphic to $\mathbb Z$ and let  $f:X \to X$ be an
    endomorphism of $\deg(f)>1$. If $\mathcal P$ is a codimension $p$ Pfaff equation invariant by $f$ then at least one of the following assertions
    hold:
    \begin{enumerate}
        \item\label{I:88} $\mathcal P$ admits a non-constant meromorphic first integral;
        \item\label{I:89} $\mathcal P$ is induced by a closed logarithmic $p$-form with poles on a totally invariant hypersurface.
    \end{enumerate}
\end{prop}
\begin{proof}
    Since $\mathcal P$ is $f$-invariant we have that,
    \begin{equation}\label{E:basico}
        f^* N_{\mathcal P} \otimes N_{\mathcal P}^* = \mathcal O_X\left(\sum_{i=1}^{k}   (\deg_{H_i}(f)-1) \cdot H_i \right)
    \end{equation}
    where $H_i$ are irreducible $\mathcal P$-invariant hypersurfaces.

    Consider the collection of hypersurfaces $\Sigma$ defined as
    \[
        \Sigma = \bigcup_{j\in \mathbb N} \bigcup_{i=1}^{k}  f^{-j}(H_i) \, .
    \]

    The  elements of $\Sigma$ are clearly invariant by $\mathcal P$.
    If $\Sigma$ is infinite then, by Jouanolou's Theorem, $\mathcal P$ admits a meromorphic  first integral and we are in
    case (\ref{I:88}).

    Otherwise, we have that $\Sigma$ is a compact hypersurface satisfying $f^{-1}(\Sigma) = \Sigma$. It follows that
    $\Sigma$ is contained in $\mathcal E$, the exceptional hypersurface  of $f$. We will assume that $f^{-1}(\mathcal E_i)=
    \mathcal E_i$ for every irreducible component of $\mathcal E$.

    Taking Chern class in (\ref{E:basico}) we obtain that
    \[
        (f^* - {\id}) c(N \mathcal P) = \sum_{i=1}^{k} (\deg_{\mathcal E_i}(f)-1) \cdot c(\mathcal E_i) \, .
    \]
    By hypothesis,  we have that $(f^* - {\rm id}) \neq 0$  and from
    the relation $f^* (\mathcal E_i) =\deg_{\mathcal E_i}(f) \cdot \mathcal E_i$ it follows that
    \[
        (f^* - {\id})^{-1}\left(\sum_{i=1}^{k} (\deg_{\mathcal E_i}(f)-1)   c(\mathcal E_i)\right) = \sum_{i=1}^{k} (\deg_{\mathcal E_i}(f)-1) \frac{ c(\mathcal E_i)}{\deg_{\mathcal E_i}(f)-1} \, .
    \]
    Thus we have just proved that
    \[
        c(N \mathcal P) = \sum_{i=1}^{k}    c(\mathcal E_i) .
    \]

    If $\omega \in {\rm H}^0(X, \Omega^p_X \otimes N\mathcal P)$ is a twisted $p$-form
    defining $\mathcal P$ then after dividing $\omega$ by the equations of $\sum_{i=1}^{k}    \mathcal E_i$ we obtain a
    rational $p$-form $\tilde \omega$ with  simple poles along the $\mathcal P$-invariant
    irreducible components of the exceptional hypersurface of $f$. The $\mathcal P$-invariance of the polar divisor of $\tilde \omega$
    implies that $d \tilde \omega$ has, at worst, simple poles. Thus $\tilde \omega$ is a logarithmic $p$-form and
    we can apply Lemma \ref{L:log log} to conclude that we are in case (\ref{I:89}).
\end{proof}

\section{Foliations of (adjoint) general type}\label{S:adjoint}
\subsection{Conventions}
We spell out some of the conventions which will be used throughout this section.
Although most of what follows is standard in the birational geometry literature,
we prefer to clarify our usage of the terms from the beginning in order to avoid misunderstandings.

Let $X$ be a normal irreducible variety with singular locus $\sing(X)$ and smooth locus $X^{\circ}$, i.e.
$X^{\circ} = X - \sing(X)$. The group $\WDiv(X)$ of Weil divisors on $X$ consists of finite sums of codimension one
irreducible subvarieties. Given a Weil divisor $D \in \WDiv(X)$, we will denote by $\mathcal O_X(D)$ the
sheaf defined over every open subset $U \subset X$ as
\[
    \mathcal O_X(D)(U) = \{ f \in \mathbb C(X) \, | \, \divisor(f) + D \ge 0\ \mbox{on}\ U \} \, .
\]
The sheaf $\mathcal O_X(D)$ is the divisorial sheaf associated to $D$, which is a reflexive rank one
sheaf over $X$. Two Weil divisors $D_1$ and $D_2$ are linearly equivalent if there exists a non-zero rational function
$f \in \mathbb C(X)^*$ such that $D_1-D_2 = (f)_0 - (f )_{\infty}$. We will denote the group of Weil divisors modulo linear equivalence by
$\Cl(X)$.

Unlike locally free sheaves, reflexive sheaves are not closed under tensor products.
One is led to consider a variant of this operation. Given two sheaves $\mathcal A$ and $\mathcal B$, the
reflexive tensor product of $\mathcal A$ and $\mathcal B$ is the double dual of the tensor product of
$\mathcal A$ and $\mathcal B$. We will denote the reflexive tensor product by $[\otimes]$. Therefore
\[
    \mathcal A [\otimes] \mathcal B = (\mathcal A \otimes \mathcal B)^{**} \, .
\]
Similarly, we will write $\mathcal A^{[m]}$ for $(\mathcal A^{\otimes m})^{**}$.
The reflexive tensor product defines a group structure on the set of rank one reflexive sheaves.
The map which sends a Weil divisor to its associated divisorial sheaf defines an isomorphism
between $\Cl(X)$, the group of Weil divisors modulo linear equivalence, and the group of isomorphisms classes
of rank one reflexive sheaves on $X$.

It will also be important to consider the group of Weil $\mathbb Q$-divisors, $\WDiv(X)\otimes \mathbb Q$. Two Weil $\mathbb Q$-divisors
$D_1$ and $D_2$ are $\mathbb Q$-linearly equivalent ($D_1 \sim_{\mathbb Q} D_2$) if there exists a non-zero integer $n$ such that $n(D_1 - D_2)$ is a Weil divisor linearly equivalent to zero.

When $X$ is singular, the sheaves $\Omega^i_X$ are not necessarily reflexive. We will denote their reflexive hulls by $\Omega^{[i]}_X$.
The tangent sheaf of $X$, defined as the dual of $\Omega^1_X$, is  reflexive.

\subsection{Foliations}\label{SS:foliations}

A foliation $\F$ on an irreducible normal variety $X$ is determined by a
saturated involutive subsheaf $T_{\F}$ of the tangent sheaf $T_X$ of $X$, $T_{\F}$ is called
the tangent sheaf of $\F$. The dimension of $\F$ is, by definition, the rank of $\TF$.
Its annihilator is a saturated subsheaf $N^*_{\F}$ of
$\Omega^{[1]}_X = T_X^* = (\Omega^1_X)^{**}$ which is called  the conormal sheaf of $\F$.
The codimension of $\F$ is the rank of $\CNF$.

We will denote the dual of $\TF$ by $\Omega^1_{\F}$.
The canonical sheaf of $\F$, denoted $\omega_{\F}$,
is $\det \Omega^1_{\F}=(\det(\TF))^*$, where $\det$ denotes the bidual of the top exterior power of a sheaf. The transverse canonical sheaf of
$\F$ is, by definition, $\omega_{X/\F} = (\det \CNF)$. 

The canonical divisor of $\F$ is, by definition, any Weil divisor $\KF$ such that
$\mathcal O_X(\KF) \simeq \omega_{\F}$. Analogously, the transverse canonical divisor
of $\F$ is, by definition, any Weil divisor $\KNF$  such that $\mathcal O_X(\KNF) \simeq \det(\CNF)$.
Although we use the definite article {\it the} to refer to $\KF$ and $\KNF$, they are not uniquely
determined as divisors, only their linear equivalence classes are.

\begin{prop}\label{P:adjunction}
    If $\F$ is a foliation on a normal irreducible variety $X$ then
    \[
        \KF + \KNF \sim K_X.
    \]
\end{prop}
\begin{proof}
    Over the  smooth locus $X^{\circ \circ} = X^{\circ} - \sing(\F)$ of $X$ and $\F$, we  have an exact sequence
    \[
        0 \to {\TF}_{|X^{\circ \circ}} \to T_{X^{\circ \circ}} \to {\NF}_{|X^{\circ \circ}} \to 0
    \]
    of locally free sheaves. The result follows by taking the determinant and extending
    the result using the normality of $X$.
\end{proof}

\subsection{Canonical singularities}\label{SS:canonical}

Given a dominant morphism $f : Y \to X$ between normal irreducible varieties
and a foliation $\F$ on $X$,  we define the pull-back of a foliation $\F$ as the foliation
$f^* \F$ on $Y$ defined by the subsheaf of $\Omega^{[1]}_Y$ determined by the saturation of the
image of the composition of the following natural morphisms
\[
    f^* \CNF \to f^* \Omega^{[1]}_X \to \Omega^{[1]}_Y \, .
\]

If $f$ is a birational morphism, $\G = f^* \F$, and $\KF + \varepsilon \KNF$ is a $\mathbb Q$-Cartier $\mathbb Q$-divisor
then the difference $f^* ( \KF + \varepsilon \KNF ) - ( K_{\G} + \varepsilon K_{Y/\G} )$
is $\mathbb Q$-linearly equivalent to a $\mathbb Q$-linear combination of exceptional divisors, i.e. we
can write
\[
    ( K_{\G} + \varepsilon K_{Y/\G} ) - f^* ( \KF + \varepsilon \KNF )  \sim_{\mathbb Q} \sum a_{\varepsilon} (E,X,\F) E \, .
\]
where $a_{\varepsilon} ( X,\F,E)$ is a rational number. The rational number $a_{\varepsilon}(E,X,\F)$ does not depend on the particular morphism $f$ but only on the exceptional divisor $E$ extracted by it.

If both $\KF$ and $\KNF$ are $\mathbb Q$-Cartier divisors then we can isolate the contributions of $\KF$ and $\KNF$ to $a_{\varepsilon}(E,X,\F)$ and write $a_{\varepsilon}(E,X,\F) = a(E,X,\F) + \varepsilon a(E,X,X/\F)$ where
\[
    K_{Y/\G}  - f^* (  \KNF )   \sim_{\mathbb Q} \sum a(E,X,X/\F) E
\]
and $a(E,X,\F) = a_0(E,X,\F)$. In this case, if we write (as usual in the birational geometry literature)
\[
    K_Y - f^*(\KX)    \sim_{\mathbb Q} \sum a(E,X) E
\]
then $a(E,X) = a(E,X,\F) + a(E,X,X/\F)$ since $K_X \sim \KF + \KNF$.

If $Z\subset X$ is an irreducible subvariety we define the $\varepsilon$-discrepancy of $Z$ to be equal to infimum of $a(E,X,\F)$
where $E$ runs over all exceptional divisors of all birational morphisms $f: Y \to X$ such that $f(E) = Z$. More concisely,
\[
    \discrep_{\varepsilon}(X,\F,Z) = \inf \{ a_{\varepsilon}(E,X,\F) \, | \, E \text{ divisor over } X \text{ with center } Z  \}.
\]

\begin{dfn}[$\varepsilon$-canonical singularities]
    Let $\F$ be a foliation on a normal irreducible variety $X$ and let $\varepsilon\ge 0$ be a non-negative real rational number.
    We will say that $\F$ has $\varepsilon$-canonical singularities along an irreducible subvariety $Z\subset X$ if
    $K_{\F} + \varepsilon \KNF$ is $\mathbb Q$-Cartier and $\discrep_{\varepsilon}(X,\F,Z) \ge 0$. We will say that a foliation
    $\mathcal F$ has $\varepsilon$-canonical singularities if it has $\varepsilon$-canonical singularities along
    every irreducible subvariety of $X$.
\end{dfn}   

\begin{rem}
    For $\varepsilon = 0$ we recover the concept of canonical singularities for foliations
    as originally defined by McQuillan, whereas for $\varepsilon=1$ we recover the homonymous concept for varieties.
\end{rem}

\begin{lemma}\label{L:propagation}
    Let $X$ be a normal irreducible variety with  canonical singularities,
    let $Z \subset X$ be an irreducible subvariety, and let $\F$ be a foliation on $X$.
    If $\F$ has $\varepsilon$-canonical singularities along $Z$ and $\KF$ is $\mathbb Q$-Cartier over an open subset intersecting $Z$ then it also has $\varepsilon'$-canonical singularities along
    $Z$ for every $\varepsilon' \ge \varepsilon$.
\end{lemma}
\begin{proof}
    Straightforward from the definitions.
\end{proof}

The definition of foliation and $\varepsilon$-canonical singularities presented in Sections~\ref{SS:foliations}
and \ref{SS:canonical} can be rephrased, mutatis mutandis, to define foliations on complex varieties as well as on formal (reduced) schemes.

\begin{lemma}\label{L:formalcriteria}
    Let $X$ be a normal irreducible variety, let $Z \subset X$ be an irreducible subvariety, and let $\F$ be a foliation on $X$.
    Let $\mathscr X$ be the formal completion of $X$ along $Z$, and let $\mathscr F$ be the  foliation on $\mathscr X$ induced by
    $\F$. If $\mathscr F$ has $\varepsilon$-canonical singularities along $Z$ then the same holds true for $\F$.
\end{lemma} %\marginpar{Il faut intervertir $\mathscr F$ et $\F$ dans l'enonce du lemme, non? I don't think so. I have added details}%
\begin{proof}
    We will prove the contrapositive. Assume $\F$ does not have $\varepsilon$-canonical singularities.
    There exists a birational morphism  $f : Y \to X$ is birational morphism and  $E$, an exceptional divisor of $f$, with $f(E) = Z$
    such that $a_{\varepsilon}(E,X,\F) <0$.
    The morphism $f$ induces a morphism of formal schemes from $\mathscr Y$, the formal completion of $Y$ along $E$, to $\mathscr X$. From the definition of $a_{\varepsilon}$, it is clear that  $a_{\varepsilon}(E,\mathscr X,\mathscr F)) = a_{\varepsilon}(E,X,\F)$. The lemma follows.
\end{proof}

As a sanity check, let us verify that smooth foliations have $\varepsilon$-canonical singularities for any $\varepsilon \ge 0$.

\begin{prop}
    Let $\F$ be a foliation of codimension $q$ on a normal irreducible variety $X$. Let $Z\subset X$ be an irreducible subvariety not contained in $\sing(X) \cup \sing(\F)$. If $Z$ is not everywhere tangent to $\F$ then $\F$ has $\varepsilon$-canonical singularities along $Z$ for every $\varepsilon \ge 0$.
\end{prop} % \marginpar{Je pense qu'il faut remplacer "not contained in" par "which do not meet". I dont't think so. According to the definition not contained in and do not meet are essentially the same thing.}
\begin{proof}
    The problem is local, so we can assume $X$ is  affine and  $\sing(X) \cup \sing(\F) = \emptyset$.
    Therefore, we can write
    \[
        0 \to \CNF \to \Omega^1_X \to \Omega^1_{\F} \to 0 \, ,
    \]
    and further restricting $X$, we can assume that the above exact sequence is a split sequence of free sheaves. In particular,
    $\KX \sim 0$, and $\omega_{\F} = \det \Omega^1_{\F}$ is generated by the restriction of a global holomorphic $q$-form $\eta \in \Omega^q_X(X)$.
    The pull-back of $\eta$ under any birational morphism $f:Y \to X$ induces a non-zero holomorphic section of the canonical sheaf of the pull-back foliation $\G = f^* \F$ with zero divisor supported on the exceptional locus of $f$. This is sufficient to show that
    $\KG$ is effective and hence $\discrep_{0}(X,\F,Z) \ge 0$. Apply Lemma \ref{L:propagation} to conclude.
\end{proof}

\subsection{Singularities of codimension one foliations}

\begin{prop}\label{P:simple properties}
    Let $\F$ be a germ of foliation at $(\mathbb C^n,0)$ with simple singularities. Then $\F$ has the properties listed below.
    \begin{enumerate}
        \item\label{I:sp-1} The tangent sheaf of $\F$ is  free. In particular, the canonical sheaf of $\F$ is trivial.
        \item\label{I:sp-2} If $\mathscr X$ is the formal completion of $X$ at $0$ and $\mathscr F$ is the foliation on $\mathscr X$ induced
        by $\F$, then the canonical sheaf of $\mathscr F$ is generated by the restriction to $T_{\mathscr F}$ of a closed logarithmic
        $(n-1)$-form with poles along invariant hypersurfaces.
        \item\label{I:sp-3} If  $f : Y \to X$ is any proper bimeromorphic morphism with center contained in $\sing(\F)$ then any exceptional
        divisor of $f$ is invariant by $f^* \F$.
    \end{enumerate}
\end{prop}
\begin{proof}
    Item (1) is a simple verification using the normal forms presented in the definition of simple singularities.
    To verify Item (2), observe that the $(n-1)$-form
    \[
        \Theta = \frac{dx_2}{x_2} \wedge \cdots \wedge \frac{dx_r}{x_r} \wedge dx_{r+1} \wedge \cdots \wedge dx_n
    \]
    is such that $\omega \wedge \Theta$ vanishes nowhere. Moreover, since $x_1, \ldots, x_r$ cut out invariant hypersurfaces, the restriction of $\Theta$ to $T_{\mathscr F}$ is regular, i.e. has no poles. It follows that $\Theta$ generates $\omega_{\mathscr F}$.
    To prove Item (3), first recall Hironaka's Chow Lemma \cite{MR1326617}[Chapter VII] which asserts that any proper bimeromorphic map is dominated by a (locally finite) composition of blow-ups. Hence, we can assume without loss of generality, that $f$ is the blow-up of an ideal supported on $\sing(\F)$. As already mentioned after Definition \ref{D:simple}, the non-resonance condition implies that the exceptional divisor is invariant, cf. \cite[Proposition 15]{MR1677402}.
\end{proof}

\begin{prop}
    If $\F$ is a codimension one foliation with simple singularities on a smooth manifold then $\F$ has canonical singularities.
\end{prop}
\begin{proof}
    Since the ambient space is smooth, $\KF$ is  Cartier.
    Let  $p \in X$ be an arbitrary closed point. Set $\mathscr X$ equal to the formal completion of $X$ at $p$,
    and set $\mathscr F$ equal to the foliation on $\mathscr X$ induced by $\F$. To check that $\F$ has canonical singularities, it suffices to do the same for $\mathscr F$ according to Lemma \ref{L:formalcriteria}. Item (\ref{I:sp-2}) of Proposition \ref{P:simple properties} provides a closed formal logarithmic $(n-1)$-form $\Theta$ with restriction to $T_{\mathscr F}$ generating $\omega_{\mathscr F}$.
    The pull-back of $\Theta$ under any bimeromorphic morphism $f:\mathscr Y \to \mathscr X$ is still a closed formal logarithmic $(n-1)$-form $\Theta$. Since the exceptional divisors of $f$ with centers contained in the polar locus of $\Theta$ are invariant by $f^*\mathscr F$ invariant (cf. Item (\ref{I:sp-3})  of Proposition \ref{P:simple properties}) the
    restriction of $f^*\Theta$ to $T_{f^* \mathscr F}$ is a non-zero section of $\omega_{f^*\mathscr F}$. It follows that $K_{f^* \mathscr F}$ is effective and, consequently, that $\mathscr F$ has canonical singularities.
\end{proof}

In general, it is unknown if every foliation is birationally equivalent to a foliation with canonical singularities. Nevertheless, in dimension two and three, there are results by Seidenberg (dimension two), Cano \cite{MR2144971} (foliations of codimension one in dimension three), and McQuillan-Panazzolo \cite{MR3128985} (foliations of dimension one in dimension three) which guarantee the existence of birationally equivalent foliations with canonical singularities. In dimension two, as well as for codimension one foliations in dimension three, there exists a birationally equivalent foliation in a smooth variety. In contrast, there exist foliations by curves of $3$-folds which are
not birationally equivalent to a foliation with canonical singularities on a smooth $3$-fold, it is unavoidable  to consider $3$-folds with cyclic quotient singularities as ambient spaces for foliations with canonical singularities. This is one of the reasons we choose to present the results of this section for foliations on singular projective varieties, instead of foliations on smooth projective varieties contrary to what we have done in the remainder of the paper.

\subsection{Action of birational maps on the pluricanonical algebra}
The lemma  below follows from standard arguments.

\begin{lemma}\label{L:standard}
    Let $f : Y \to X$ be a birational morphism  between normal projective algebraic varieties.
    If $\F$ is a foliation on $X$ with  $\varepsilon$-canonical singularities
    and $\G$ denotes the foliation $f^*\F$ on $Y$ then there is a natural
    isomorphism
    \[
        f^* : H^0(X, \mathcal O_X( a K_{\F} +  b \KNF)) \to H^0(Y, \mathcal O_Y( a K_{\G} + b K_{Y/\G} ))
    \]
    whenever  $a$ and $b$ are positive integers subject to the equality $b = \varepsilon a$.
\end{lemma}
\begin{proof}
   From the definition of $\varepsilon$-canonical singularities we have that
    \[
        \left( K_{\mathcal G} + \varepsilon K_{Y/\mathcal G} \right) - f^* \left( K_{\F} + \varepsilon \KNF \right) = \sum a_i E_i
    \]
    where $a_i$ are positive rational numbers and $E_i$ are $f$-exceptional divisors. If we multiply this identity by
    $a = b/\varepsilon$, we deduce that
    \[
        \left(  a K_{\mathcal G} + b K_{Y/\mathcal G} \right) - f^* \left(  a K_{\F} + b \KNF \right)
    \]
    is an effective $\mathbb Q$-divisor. Therefore, the pull-back of rational functions $f^* : \mathbb C(X)  \to \mathbb C(Y)$ defines
    a natural linear inclusion
    \[
        f^* : H^0(X, \mathcal O_X( a K_{\F} +  b \KNF)) \to H^0(Y, \mathcal O_Y( a K_{\mathcal G} + b K_{Y/\mathcal G} )).
    \]
    Since $f$ is a birational morphism and $X$ is normal, we also have a natural linear inclusion
    \[
        f_*  : H^0(Y, \mathcal O_Y( a K_{\mathcal G} + b K_{Y/\mathcal G} )) \to H^0(X, \mathcal O_X( a K_{\F} +  b \KNF))
    \]
    obtained by restricting a section to the complement of the exceptional locus of $f$, pushing the result down using $f$, and
    extending it using the normality of $X$. Since the composition of these two inclusions is nothing but the identity on the vector
    space $H^0(X, \mathcal O_X( a K_{\F} +  b \KNF))$, we deduce that $f^*$ is an isomorphism.
\end{proof}

\begin{prop}\label{P:inducedaction}
    Let $f : Y \dashrightarrow X$ be a birational map between normal projective algebraic varieties.
    If $\F$ and $f^*\F= \G$ are  foliations  with  $\varepsilon$-canonical singularities
    on $X$ and $Y$  then there is a natural
    isomorphism
    \[
        f^* : H^0(X, \mathcal O_X( a K_{\F} +  b \KNF)) \to H^0(Y, \mathcal O_Y( a K_{\G} + b K_{Y/\G} ))
    \]
    whenever  $a$ and $b$ are positive integers subject to the equality $b = \varepsilon a$.
\end{prop}
\begin{proof}
    Let $\pi: Z \to Y$ be a resolution of indeterminacies of $f : Y \dashrightarrow X$.
    Concretely, $Z$ is a normal projective variety and  $\pi: Z \to Y$ is a
    birational morphism such that $f \circ \pi : \pi^{-1}(Y) \dashrightarrow X$
    can be extended to a birational morphism from $Z$ to $X$ which we will denote by $g : Z \to X$.

    \begin{center}
    \begin{tikzcd}[column sep=small]
    & Z \arrow[dl, "\pi" left] \arrow[dr,"g"] & \\
    Y \arrow[rr,dashed, "f"] & & X
    \end{tikzcd}
    \end{center}

    Set $\mathcal H  = g^* \mathcal F$. Apply Lemma \ref{L:standard} to get
    isomorphisms
    \[
        g^* : H^0(X, \mathcal O_X( a K_{\F} +  b \KNF)) \to H^0(Z, \mathcal O_Z( a K_{\mathcal H} + b K_{Z/\mathcal H} )).
    \]
    and
    \[
        \pi^* : H^0(Y, \mathcal O_Y( a K_{\G} +  b K_{Y/G})) \to  H^0(Z, \mathcal O_Z( a K_{\mathcal H} + b K_{Z/\mathcal H})).
    \]
    The result follows by taking  $f^*$ equal to $(\pi^*)^{-1} \circ g^*$.
\end{proof}

\subsection{Action of rational maps on the pluricanonical algebra} The results of the previous section
have versions for dominant rational maps between varieties of the same dimension. The only new ingredient
needed is the following lemma. A version of it for foliations on surfaces was proved in \cite[Proposition 2.1]{MR2818727}.

\begin{lemma}\label{L:quasidef}
    Let $f : Y \dashrightarrow X$ be a dominant rational map, not necessarily birational,
    between normal projective algebraic varieties of the same dimension.
    If $\F$ is a foliation on $X$ with  $\varepsilon$-canonical singularities
    and $\G$ denotes the foliation $f^*\F$ on $Y$ then
    $\KG + \varepsilon K_{Y/\G} - f^*(\KF + \varepsilon \ K_{X/\F})$ is an effective $\mathbb Q$-divisor.
\end{lemma}
\begin{proof}
    If $U \subset X$ is an open subset of the regular locus of $X$ where  $f$ is a local biholomorphism then  $f^* {\omega_{\F}}_{|U}$ and ${f^* \omega_{X/\F}}_{|U}$ are clearly isomorphic to ${\omega_{\G}}_{|U}$ and ${\omega_{Y/\G}}_{|U}$. Thus the $\mathbb Q$-divisor $\KG + \varepsilon K_{Y/\G} - f^*(\KF + \varepsilon \ K_{X/\F})$
    is linearly equivalent to a divisor $\Delta$ supported on the critical divisor of $f$. Let $E$ be one of its irreducible
    components. We want to show that $\ord_E \Delta \ge 0$.

    First assume that $f$ does not contract $E$, i.e. $\dim f(E) = \dim E$. To compute $\ord_E \Delta$, we can localize at
    a neighborhood of a sufficiently general point of $E$ and choose coordinates at neighborhoods of $x \in X$ and
    of $y = f(x) \in Y$ such that the map $f$ takes to the form $f(x_1, \ldots, x_{n-1}, z) = (x_1, \ldots, x_{n-1},z^m)$
    for some $m \ge 2$. Note that $\ord_E(K_Y - f^* K_X) = m-1$.

    Let $q= \codim (\F)$ and let $\omega$ be a germ of $q$-form at $y$ defining $\F$ and without codimension one zeros. We can write
    \[
        \omega = \sum_{i=0}^\infty \alpha_i z^i + dz\wedge\left(\sum_{i=0}^\infty z^i\beta_i \right)
    \]
    where the $\alpha_i$'s are $q$-forms on the variables $x_1, \ldots, n-1$ and
    the $\beta_i$'s are $(q-1)$-forms of the same type.

    If $E$ is invariant by $\F$ then $\alpha_0=0$ and, because
    $\omega$ has no codimension one zeros,  $\beta_0\neq 0$. A direct compuation shows that $z^{m-1}$ divides
    $f^* \omega$ while $z^m$ does not divide it.
    It follows that $\ord_E( \omega_{Y/\G} - f^* \omega_{X/\F} ) = m-1$ and, by adjunction, $\ord_E( \omega_{\G} - f^* \omega_{\F} ) = 0$.

    If $E$ is not invariant by $\F$ then $\alpha_0 \neq 0$. Consequently, $f^*\omega$ does not have codimension one
    zero along $E=\{ z=0\}$. Therefore $\ord_E( \omega_{Y/\G} - f^* \omega_{X/\F} ) = 0$ and, by adjunction, $\ord_E( \omega_{\G} - f^* \omega_{\F} ) = m-1$.

    No matter if $E$ is $\mathcal F$-invariant or not, we have that $\ord_E \Delta \ge 0$ as soon as $E$ is not contracted.
    Notice that we have not used the hypothesis on the nature of the singularities of $\F$ yet.

    Suppose from now on that $\dim f(E)< \dim E$, i.e. $E$  is contracted by $f$. Set $X_0 = X$ and $f_0 = f$.
    Assume $X_i$ and $f_i : Y \dashrightarrow X_i$ defined. Let $\pi_{i+1} : X_{i+1} \to X_i$ be the normalization of the  blow-up of $X_i$ along $f_i(E)$ and let $f_{i+1} : Y \dashrightarrow X_{i+1}$ be the composition $\pi_{i+1}^{-1} \circ f_i$.
    According to \cite[Theorem VI.1.3]{MR1440180}, there exists a $k\ge 0$ such that $f_k(E)$ is a divisor of $X_k$.
    Let $\pi = \pi_1\cdots\circ\pi_{k-1}\circ \pi_k : X_k \to X$ be the composition of the blow-ups. We can write
    $K_{\G} + \varepsilon  K_{Y/\G} - f^*(K_{\F} + \varepsilon  K_{X/\F})$ as
    \begin{align*}
         & \left(K_{\G} + \varepsilon  K_{Y/\G} - f_k^*\left(K_{\F_k} + \varepsilon  K_{X_k/\F_k} \right) \right) \\
        + & f_k^*\left( K_{\F_k} + \varepsilon  K_{X_k/\F_k} - \pi_k^*\left( \KF + \varepsilon \KNF \right)  \right) \, .
    \end{align*}
    The expression in the first line has positive order along $E$ by the argument above. The expression in the second line  also has positive order along $E$ since we are assuming that $\F$ has $\varepsilon$-canonical singularities. This concludes the proof of the lemma.
\end{proof}

\begin{prop}\label{P:inducedaction2}
    Let $f : Y \dashrightarrow X$ be a dominant rational map between normal projective algebraic varieties of the same dimension.
    If $\F$ is a  foliation  with  $\varepsilon$-canonical singularities
    on $X$ and $\G= f^* \F$ is the induced foliation on $Y$  then there is a natural
    inclusion
    \[
        f^* : H^0(X, \mathcal O_X( a K_{\F} +  b \KNF)) \to H^0(Y, \mathcal O_Y( a K_{\G} + b K_{Y/\G} ))
    \]
    whenever  $a$ and $b$ are positive integers subject to the equality $b = \varepsilon a$.
\end{prop}
\begin{proof}
    Similar to the proof of Proposition \ref{P:inducedaction} using Lemma \ref{L:quasidef} in place of the
    definition of $\varepsilon$ canonical singularities.
\end{proof}

\begin{cor}\label{C:contra}
    If $\G$ is a foliation on a normal projective variety $Y$ which is birationally equivalent to a foliation
    $\F$ with  $\varepsilon$-canonical singularities then
    \[
        h^0(Y, \mathcal O_Y( a K_{\G} + b K_{Y/\G} )) \ge h^0(X, \mathcal O_X( a K_{\F} +  b \KNF))
    \]
    whenever  $a$ and $b$ are positive integers subject to the equality $b = \varepsilon a$.
\end{cor}

\subsection{Foliations invariant by algebraic actions}

\begin{lemma}\label{L:action}
    Let $\F$ be a foliation on a normal projective variety $X$ invariant by a non-constant  algebraic
    action $\varphi : G \times X \to X$ with  $G=(\mathbb C^*,\cdot)$ or $G=(\mathbb C,+)$.
    Then $\F$ is birationally equivalent to a foliation $\G$ on a smooth projective variety $Y$ such that
    $\KG + \varepsilon K_{Y/\G}$ is not pseudo-effective for any $\varepsilon \in (0,1)$.
\end{lemma}
\begin{proof}
    Thanks to the existence of equivariant resolution of singularities  [K\'ollar, Proposition 3.9.1]
    there is no loss of generality in assuming that $X$ is smooth.

    If $\F$ is a uniruled foliation then $\F$ is birationally equivalent to a foliation
    $\G$ on a smooth variety $Y$ with $\KG$ not pseudo-effective according to the proof of \cite[Theorem 3.7]{MR3842065}.
    Since $Y$ is smooth $K_Y$ is also not pseudo-effective. It follows that $\KG + \varepsilon K_{Y/\G} = (1-\varepsilon) \KG + \varepsilon \K_Y$ is not pseudo-effective for any $\varepsilon \in [0,1)$.

    Assume from now on that $\F$ is not uniruled.
    Let $\mathcal A$ be the one-dimensional foliation defined by the action $\varphi$.
    Note that Zariski closure of every leaf of $\mathcal A$ is a rational curve.
    Let $C$ be the general leaf of $\mathcal A$. As we are assuming that $\F$ is not uniruled, $C$ is generically transverse to $\F$.

    Let us first consider the particular case where $C$ is compact. In this case, a neighborhood $C$ in $X$ is isomorphic
    to $\mathbb D^{n-1} \times \mathbb P^1$ and $\mathcal A$ is defined by the projection
    $\mathbb D^{n-1} \times \mathbb P^1 \to \mathbb D^{n-1}$. Moreover, we can assume that
    the action $\varphi$ admits one of the following local normal forms in this neighborhood:
    \[
        \varphi(t,(x,y)) =
        \left\{
            \begin{array}{ll}
                (x, t \cdot y)  & {\text{if } \quad G= (\mathbb C^*,\cdot)  ;} \\
                (x, y+t )       & {\text{if } \quad G  = ( \mathbb C,+).}
            \end{array}
        \right.
    \]
    Since $C$ is general, the foliation $\F$ is smooth in this neighborhood. If $\omega$ is
    a $q$-form  defining $\F$ then we can write
    \[
        \omega = \sum_{|I|=q, i\ge 0} a_{I,i}(x) y^i dx^I + \sum_{|J|=q-1,i\ge 0} b_{J,i}(x) y^i dx^J\wedge dy \, ,
    \]
    in such a way that any common factor of all $a_{I,i}$ and $b_{J,i}$ does not depend on $y$.
    Since $C$ is generically transverse to $\F$, at least one of the $b_{j,i}$'s is non-zero. Since $\F$ is smooth,
    we can assume that $a_{I,0}\neq0$ for some $I$ or $b_{J,0}\neq 0$ for some $J$. Combining the  $\varphi$-invariance of
    $\F$ with these two remarks, we deduce that $\omega$ can be written as
    \[
        \displaystyle{\sum_{|I|=q} a_{I,1}(x) y dx^I + \sum_{|J|=q-1} b_{J,0}(x) dx^J \wedge dy},
    \]
    when $G=(\mathbb C^*, \cdot)$, and
    \[
        \displaystyle{\sum_{|I|=q} a_{I,0}(x) dx^I + \sum_{|J|=q-1} b_{J,0}(x) dx^J \wedge dy}
    \]
    when $G=(\mathbb C,+)$. In both  cases,  $\KNF \cdot C = -2$ and, by adjunction,
     $\KF\cdot C=0$. Since $C$ is part of a covering family of curves, we conclude that $\KF + \varepsilon \KNF$
    is not pseudo-effective for any $\varepsilon>0$.

    To deal with the general non-compact case, start by observing that the singular set of $\mathcal A$ is contained in the set of fixed points of $\varphi$. Let $\mathcal C$ be the generic (in the schematic sense) leaf of $\mathcal A$ and let $Z$ be the closure of the intersection of $\mathcal C$ with $\sing(\mathcal A)$. Replace $X$ by $\Bl_Z X$, the blow-up of $X$ along $Z$. Since
    $Z$ is contained in the set of fixed points of $\varphi$, the action lifts to $\Bl_Z X$. Replace $\Bl_Z X$ by an equivariant desingularization of it. If the general leaf of the foliation by curves determined by the lift of $\varphi$ is compact, apply the argument of the previous paragraph to conclude. Otherwise, repeat the argument of the current paragraph. This procedure will stop after finitely many blow-ups, thanks to \cite[Theorem VI.1.3]{MR1440180}.
\end{proof}

\subsection{Foliations of adjoint general type}

\begin{thm}\label{T:adj gen type}
    Let $\varepsilon \ge 0$ be a non-negative rational number, and let $\F$ be a foliation with
    $\varepsilon$-canonical singularities on a normal projective variety $X$. If
    \[
        \limsup_{N} \frac{\log h^0(X , \mathcal O_X( N (\KF + \varepsilon \KNF)))}{\log N} = \dim X
    \]
    then the group of birational automorphisms of $\F$ is finite. Moreover, any rational endomorphism of $\F$ is birational.
\end{thm}
\begin{proof}
    Let $N \gg 0$ be sufficiently divisible, and set $V=H^0(X, \mathcal O_X( N ( \KF + \varepsilon \KNF)))$.
    The rational map
    \begin{align*}
        \phi : X & \dashrightarrow \mathbb P \left( V^*\right) \\
        x &\mapsto \{ s \in V ; s(x) = 0 \}
    \end{align*}
    is birational onto its image, see for instance \cite[Theorem 2.1.33]{MR2095471}. Let $Z  \subset \mathbb P(V^*)$ be the closure of the image of $\phi$.
    Since $Z$ is birational to $X$, we have a foliation $\mathcal G = (\phi_N)_*\F$ on $Z$ too.

    If $f : X \dashrightarrow X$ is a birational automorphism of $\F$ then the projectivization of the dual of the automorphism
    $f^*$ provided by Proposition \ref{P:inducedaction} is an
    automorphism $[f_*]$ of  $\mathbb P( V^*)$ which preserves $Z$ and the foliation $\G$ and fits into the commutative diagram below.
    \[
    \begin{tikzcd}[column sep=small]
        X   \arrow[d, dashed, "\phi" left] \arrow[rr,dashed, "f"] && X  \arrow[d, dashed, "\phi" ] \\
        Z \arrow[rr, "(f_*)_{|Z}"] \arrow[d,hook]& & Z \arrow[d,hook] \\
        \mathbb P \left( V^*\right) \arrow[rr, "f_*"] & & \mathbb P \left( V^*\right)
    \end{tikzcd}
    \]

    It follows the existence of a group homomorphism $\Bir(X,\F) \to \Aut(\mathbb P(V^*))$.
    Since $\phi : X \to \mathbb P(V^*)$ is birational, this homomorphism is injective and has image
    equal to the  subgroup formed by automorphisms of $\mathbb P(V^*)$ preserving
    $Z$ and the foliation $\G$.

    Aiming at a contradiction, assume $\Bir(X,\F)$ is infinite. Since it is identified with a closed
    subgroup of a linear algebraic group it follows that there is an algebraic one-parameter subgroup
    contained in it. Concretely, there exists an algebraic action of $\mathbb C$ or $\mathbb C^*$ on $Z$
    which preserves the foliation $\G$. Lemma \ref{L:action} implies that $\G$ is birationally equivalent
    to a foliation $\mathcal H$ on a smooth projective variety $Y$ for which $K_{\mathcal H} + \varepsilon K_{Y/\mathcal H}$
    is not pseudo-effective for any $\varepsilon \in (0,1)$. In particular,
    \[
        h^0(Y,N (K_{\mathcal H} + \varepsilon K_{Y/\mathcal H})) = 0
    \]
    for every $N>0$. This contradicts Corollary \ref{C:contra}. The theorem follows.
\end{proof}

\bibliography{references}{}
\bibliographystyle{alpha}

\end{document}